\documentclass{article}
\usepackage{a4wide,amsmath,amsthm,epsfig,graphicx}
\usepackage{amsmath,amsthm,amssymb}
\usepackage{amsfonts}
\usepackage{graphics}

\usepackage{dsfont}
\usepackage{multirow,url}

\newtheorem{thm}{Theorem}[section]

\newtheorem{lemma}[thm]{Lemma}
\newtheorem{prop}[thm]{Proposition}

\newtheorem{remark}{Remark}

\newcommand{\rd}{\mathrm d}
\newcommand{\rE}{\mathrm E}
\newcommand{\ts}{\mathrm{TS}^p_\alpha}
\newcommand{\Ga}{\mathrm{Ga}}

\newcommand{\IBGM}{\mathrm{IBGM}}

\newcommand{\GGa}{\mathrm{GGa}}
\newcommand{\DGGa}{\mathrm{DGGa}}
\newcommand{\IGa}{\mathrm{IGa}}

\newcommand{\siid}{\stackrel{\mathrm{iid}}{\sim}}
\newcommand{\eqd}{\stackrel{d}{=}}

\newcommand{\sde}{\emph{SDE}}
\newcommand{\refeqq}[1]{~(\ref{#1})}
\newcommand{\myref}[1]{~\ref{#1}}
\newcommand{\mycite}[1]{~\cite{#1}}

\newcommand{\pdf}{\textit{pdf}}
\newcommand{\pmf}{\textit{pmf}}
\newcommand{\cdf}{\textit{cdf}}

\usepackage[a4 paper, top = 1.3 cm, bottom = 1.5 cm, left = 1.3 cm, right = 1.3 cm]{geometry}
\title{\Huge \textbf{Efficient Simulation of 
$p$-Tempered $\alpha$-Stable OU Processes}}

\author{Michael \textsc{Grabchak}\footnote{mgrabcha@uncc.edu}  \\
University of North Carolina Charlotte\\
Piergiacomo
\textsc{Sabino}\footnote{piergiacomo.sabino@eon.com}
\footnote{The views, opinions, positions or strategies expressed in this article are those of the authors
and do not necessarily represent the views, opinions, positions or strategies of, and should not be
attributed to E.ON SE.}
\\
Quantitative Risk Management, E.ON SE\\
\vspace{5pt}
 Br\"{u}sseler Platz 1, 45131 Essen, Germany and\\
Department of Mathematics and Statistics, \\University of Helsinki P.O.\ Box 68 FI-00014 Finland\\
}

\begin{document}
    \maketitle \thispagestyle{empty}

\begin{abstract}
We develop efficient methods for simulating processes of Ornstein-Uhlenbeck type related to the class of $p$-tempered $\alpha$-stable ($\ts$) distributions. Our results hold for both the univariate and multivariate cases and we consider both the case where the $\ts$ distribution is the stationary law and where it is the distribution of the background driving L\'evy process (BDLP). In the latter case, we also derive an explicit representation for the transition law as this was previous known only in certain special cases and only for $p=1$ and $\alpha\in[0,1)$.  Simulation results suggest that our methods work well in practice.
\end{abstract}

\section{Introduction}\label{sec:introduction}

Tempered stable processes of Ornstein Uhlenbeck type have been the subject of much research in recent years. They combine two important directions. First, they are examples of non-Gaussian processes of Ornstein Uhlenbeck type (OU processes), which have a more intricate dependence structure than the more commonly used L\'evy processes and, in particular, they are mean reverting. In financial applications this makes them natural models for various quantities including stochastic volatility, stochastic interest rates, and commodity prices, see, e.g., \cite{BNSh01}, \cite{BMBK07}, \cite{Sabino20} and \cite{Sabino21}. Second, they are based on the class of tempered stable (TS) distributions, which has been gaining in prominence over the last decade. These distributions can approximate the more common Gaussian and stable distributions, but their tail behavior is more realistic, which makes them useful for a variety of application areas. We are particularly motivated by their use in the modeling of financial returns, see \cite{Grabchak:Samorodnitsky:2010}, \cite{Xia:Grabchak:2022}, and the references therein. TS distributions were first formalized in the classic paper \cite{Rosinski07}. Since then, they have been extended in several directions in \cite{Rosinski:Sinclair:2010}, \cite{Bianchi:etal:2011}, and \cite{Grabchak12}, see also the monograph \cite{Grabchak16}.

There are two types of OU processes related to TS distributions. The first, denoted TSOU processes, correspond to the case where the stationary distribution is TS. The second, denoted OUTS processes, correspond to the case where the background driving L\'evy process (BDLP) has a TS distribution. The study of the transition laws of TSOU processes has been primarily focused on the fairly simple univariate class of so-called classical tempered stable (CTS) distributions, see \cite{Zhang:Zhang:2009}, \cite{Zhang11}, \cite{KawaiMasuda:2012}, \cite{QDZ21}, \cite{cs20_3}, and the references therein. More detailed results for the special cases of gamma and inverse Gaussian distributions are given in \cite{Zhang:Zhang:2008}, \cite{QDZ19}, and \cite{SabinoCufaro20}. Extensions to general classes of TS distributions, including the multivariate case, are given in \cite{Grabchak20} and \cite{Grabchak21}. Significantly less attention has been paid to the class of OUTS processes. To the best of our knowledge, in this case, the transition laws have only been studied in the case of gamma distributions in \cite{QDZ19} and \cite{SabinoCufaro20}, and in the case of CTS distributions with parameter $\alpha\in[0,1)$ in \cite{QDZ21} and \cite{cs20_3}. Although, it should be noted that some preliminary results about certain univariate TS distributions beyond CTS are given in \cite{QDZ21}.

The purpose of the current paper is two-fold. First, we derive the transition laws of OUTS processes for the class of $p$-tempered $\alpha$-stable ($\ts$) distributions with any $\alpha\in(-\infty,2)$ and $p>0$. These distributions form a large and flexible class of both univariate and multivariate models and include CTS distributions as a special case. Our results compliment those in \cite{Grabchak21} for the corresponding class of TSOU processes. Second, we develop efficient methods for simulating from the transition laws of both TSOU and OUTS processes based on $\ts$ distributions. These can then be used to simulate the corresponding OU process on a finite grid. Our simulation methods extend the ideas introduced in \cite{cs20_3} for CTS distributions and work well in practice. The main idea is based on showing that certain components of the transition law can be represented as generalized gamma scale mixtures (GGSMs). For this reason we develop the theory of such mixtures and give multiple approaches for simulation.

The rest of the paper is organized as follows. In Section \ref{sec: GGSM} we introduce GGSM distributions and discuss several examples that are important for simulating TSOU and OUTS processes. In Section \ref{sec: TS} we recall the definition of $\ts$ distributions and give some properties. Then in Section \ref{sec: OU processes} we recall basic properties of OU processes and, in particular, we give our results on the transition laws of OUTS processes. In Section \ref{sec:num:exp} we perform a series of numerical experiments to better understand the performance of our simulation methods. Proofs are postponed to Section \ref{sec: proofs}.

Before proceeding we introduce some notation. We write \cdf, \pdf, and \pmf\ for \emph{cumulative distribution function}, \emph{probability density function}, and \emph{probability mass function}, respectively. For a distribution $F$ we write $X\sim F$ to denote that $X$ is a random variable with distribution $F$ and we write $X_1,X_2,\dots\siid F$ to denote that $X_1,X_2,\dots$ are independently and identically distributed random variables with distribution $F$. For simplicity, instead of $F$, we sometimes write the corresponding \pdf\ or \pmf. We write $U(0,1)$ to denote a uniform distribution on $(0,1)$, $\delta_a$ to denote a point-mass at $a$, and $1_A$ to denote the indicator function on $A$. We write $\vee$ and $\wedge$ to denote the maximum and minimum, respectively, and we write $\lfloor\cdot\rfloor$ to denote the floor function. Further, we use the convention that $\sum_{n=1}^0$ is $0$. We write $\eqd$ and $=:$ to denote equality in distribution and a defining equality, respectively.

\section{Generalized Gamma Scale Mixtures}\label{sec: GGSM}

In this section we introduce the class of generalized gamma scale mixture (GGSM) distributions and discuss various properties and approaches for simulation. We then show that the incomplete gamma (IGa) distribution, which was introduced in \cite{Grabchak21} as an important component of the transition law of a TSOU process, is a GGSM and use this fact to develop efficient simulation techniques. Next, we introduce two new GGSM distributions, which play a similar role in the study of OUTS processes. We note that the results in this section may be of independent interest.

We begin by recalling that the generalized gamma distribution was introduced in \cite{Stacy:1962} and has a \pdf\ given by
$$
g_{\gamma,p,\theta}(u) = \frac{p\theta^{\gamma/p}}{\Gamma(\gamma/p)}u^{\gamma-1}e^{-u^p \theta},\ u>0,
$$
where $\gamma,p,\theta>0$ are parameters. We denote this distribution by $\GGa(\gamma,p,\theta)$. When $p=1$ it reduces to the usual gamma distribution, which we denote by $\Ga(\gamma,\theta)$. 
We can simulate from $\GGa(\gamma,p,\theta)$ by using the fact that
\begin{eqnarray}\label{eq: sim GGa}
\mbox{if }X\sim \Ga(\gamma/p,1)\mbox{, then }\left(X/\theta\right)^{1/p}\sim \GGa(\gamma,p,\theta),
\end{eqnarray}
see \cite{Grabchak20} and \cite{Grabchak21} for more on this distribution.

A GGSM distribution has a \pdf\ of the form
\begin{eqnarray}\label{eq: GGSM defn}
f(u) = \int_0^\infty  g_{\gamma,p,\theta^p}(u) m(\theta)\rd \theta, 
\end{eqnarray}
where $m$ is the \pdf\ of the so-called mixing distribution, whose support is contained in $[0,\infty)$. Note that, in \eqref{eq: GGSM defn}, the parameter is $\theta^p$ and not just $\theta$. To simulate from a GGSM we can first simulate $Z\sim m$ and then, given $Z$, simulate $X\sim  \GGa(\gamma,p,Z^p)$. Equivalently, using \eqref{eq: sim GGa}, we get the following Algorithm.\\

\noindent \textbf{Algorithm GGSM1.}\\
\textbf{Step 1.} Independently simulate $Y\sim \Ga(\gamma/p,1)$ and $Z\sim m$.\\
\textbf{Step 2.} Return $Y^{1/p}/Z$.\\ 
 
In practice, it is not always easy to simulate from $m$. Under general assumptions, which always hold in the situations that are of interest to us, we can set up a rejection sampling approach by using the following result, which follows immediately from \eqref{eq: GGSM defn}.

\begin{lemma}
If the support of $m$ is lower bounded by some $a>0$ and if $\int_a^\infty \theta^\gamma m(\theta)\rd \theta<\infty$, then
$$
f(u) \le V g_{\gamma,p,a^p}(u),
$$
where $V=a^{-\gamma} \int_a^\infty \theta^\gamma m(\theta)\rd \theta$.
\end{lemma}

Let 
$$
\varphi(u) = \frac{ \int_a^\infty e^{-u(\theta^p-a^p)}\theta^\gamma m(\theta)\rd \theta}{ \int_a^\infty \theta^\gamma m(\theta)\rd \theta}.
$$

\noindent \textbf{Algorithm GGSM2.}\\
\textbf{Step 1.} Independently simulate $U\sim U(0,1)$ and $Y\sim \Ga(\gamma/p,1)$.\\
\textbf{Step 2.} If $U\le \varphi(Y/a^p)$ return $Y^{1/p}/a$, otherwise go back to step 1.\\ 

From standard results about rejection sampling algorithms, the probability of rejection on a given iteration is given by $1/V$.

\subsection{Incomplete Gamma Distribution}\label{sec:iga:dist}

For $u>0$ and $\gamma>0$, let
$$
G_\gamma(u) = \frac{1}{\Gamma(\gamma)}\int_0^u x^{\gamma-1}e^{-x}\rd x
$$
be the scaled lower incomplete gamma function. It is the \cdf\ of the $\Ga(\gamma,1)$ distribution.  The incomplete gamma (IGa) distribution has a \pdf\ given by
$$
f_{\beta,\gamma,p,\eta}(u) = \frac{1}{K_{\beta,\gamma,p,\eta}} G_\gamma(u^p(\eta-1))e^{-u^p}u^{-1-\beta}, \qquad u>0,
$$
where $\gamma>0$, $p>0$, $\eta>1$, $\beta\in(-\infty,p\gamma)$, and $K_{\beta,\gamma,p,\eta}$ is a normalizing constant given by
$$
K_{\beta,\gamma,p,\eta} = \frac{\Gamma(\gamma-\beta/p)}{p\Gamma(\gamma)} K^*_{\beta,\gamma,p,\eta}
$$
with
\begin{eqnarray*}
K^*_{\beta,\gamma,p,\eta}=\int_{1/\eta}^1(1-x)^{\gamma-1}x^{-\beta/p-1}\rd x 
= \int_{1}^\eta (x-1)^{\gamma-1}x^{\beta/p-\gamma}\rd x .
\end{eqnarray*}
We denote this distribution by $\IGa(\beta,\gamma,p,\eta)$. In \cite{Grabchak21} it was shown that if $W\sim \IGa(\beta,\gamma,p,\eta)$ and $\xi>\beta-p\gamma$, then
\begin{eqnarray}\label{eq: moments of IGa}
\rE\left[W^\xi\right] = \frac{K_{(\beta-\xi),\gamma,p,\eta} }{K_{\beta,\gamma,p,\eta} }.
\end{eqnarray}
Further, Proposition 1 in that paper shows that
\begin{eqnarray}\label{eq: K* asymp}
K^*_{\beta,\gamma,p,\eta}\sim \frac{(\eta-1)^\gamma}{\gamma}\mbox{  as  }\eta\downarrow1.
\end{eqnarray}
We now show that the IGa distribution is a GGSM. When $p=1$ and $\gamma=1$, this was already observed in \cite{cs20_3}. 

\begin{lemma}\label{lemma: IGa is GGSM}
We have
\begin{eqnarray*}
f_{\beta,\gamma,p,\eta}(u) =\int_1^{\eta^{1/p}}  g_{(p\gamma-\beta),p,\theta^p}(u) m_{\beta,\gamma,p,\eta}(\theta)\rd \theta,
\end{eqnarray*}
where
\begin{eqnarray}\label{eq: m for IGa}
m_{\beta,\gamma,p,\eta}(\theta) = \frac{p}{K^*_{\beta,\gamma,p,\eta}} (\theta^p-1)^{\gamma-1}\theta^{p+\beta-p\gamma-1}, \ \ 1<\theta<\eta^{1/p}.
\end{eqnarray}
\end{lemma}

Since IGa is a GGSM, we can use Algorithms GGSM1 and GGSM2 to simulate from it. It is readily checked that if $Z\sim m_{\beta/p,\gamma,1,\eta}$, then $Z^{1/p}\sim m_{\beta,\gamma,p,\eta}$. This implies that Algorithm GGSM1 reduces to the following.\\

\noindent \textbf{Algorithm IGa1.} Simulation from $\IGa(\beta,\gamma,p,\eta)$.\\
\textbf{Step 1.} Independently simulate $Y\sim \Ga(\gamma-\beta/p,1)$ and $Z\sim m_{\beta/p,\gamma,1,\eta}$.\\
\textbf{Step 2.} Return $(Y/Z)^{1/p}$.\\

Next, we note that for the IGa distribution Algorithm GGSM2 reduces to the algorithm introduced in \cite{Grabchak21}. It can be stated as follows. Let
\begin{eqnarray*}
\varphi_1(u) &=&  \frac{ \int_1^{\eta^{1/p}} e^{-u(\theta^p-1)} (\theta^p-1)^{\gamma-1}\theta^{p-1}\rd \theta}{ \int_1^{\eta^{1/p}} (\theta^p-1)^{\gamma-1}\theta^{p-1}\rd \theta}
= \frac{ \Gamma(\gamma+1)}{  (\eta-1)^{\gamma}} G_\gamma((\eta-1)u) u^{-\gamma}.
\end{eqnarray*}

\noindent \textbf{Algorithm IGa2.} Simulation from $\IGa(\beta,\gamma,p,\eta)$.\\
\textbf{Step 1.} Independently simulate $U\sim U(0,1)$ and $Y\sim \Ga(\gamma-\beta/p,1)$.\\
\textbf{Step 2.} If $U\le \varphi_1(Y)$ return $Y^{1/p}$, otherwise go back to step 1.\\

 In \cite{Grabchak21} it is shown that, on a given iteration, the probability of acceptance is $1/V_1$, where $V_1 = \frac{(\eta-1)^\gamma}{\gamma K^*_{\beta,\gamma,p,\eta}}$. From \eqref{eq: K* asymp} it follows that $1/V_1\to1$ as $\eta\downarrow1$. As we will see, when simulating TSOU processes, we typically take $\eta$ close to $1$.

In order to use Algorithm IGa1, we need a way to simulate from $m_{\beta,\gamma,1,\eta}$. We will provide several algorithms for doing this. First, we introduce the \pdf
$$
\ell_{\delta,\eta} (\theta) = \frac{\delta +1}{\eta^{\delta +1} -1}\theta^{\delta} \quad 1<\theta<\eta,
$$
where $\delta\in\mathbb R$ and $\eta>1$ are parameters. Here and throughout we interpret  $\frac{\delta +1}{\eta^{\delta +1} -1}$ by its limiting value of $1/\ln\eta$ when $\delta=-1$. It is easy to check that we can simulate from this distribution as follows.\\

\noindent \textbf{Algorithm $\ell$1.} Simulation from $\ell_{\delta,\eta}$.\\
\textbf{Step 1.} Simulate $U\sim U(0,1)$.\\
\textbf{Step 2.} If $\delta=-1$ return $\eta^U$. Otherwise, return $\left[1 + U\,\left(\eta^{\delta+1} -1\right)\right]^{\frac{1}{\delta+1}}$.\\

Note that, when $\gamma=1$, we have $m_{\beta,1,1,\eta}=\ell_{\beta-1,\eta}$, which leads to the following algorithm.\\

\noindent \textbf{Algorithm M0.} Simulation from $m_{\beta,1,1,\eta}$.\\
\textbf{Step 1.} Simulate $U\sim U(0,1)$.\\
\textbf{Step 2.} If $\beta=0$ return $\eta^U$. Otherwise, return $\left(U(\eta^{\beta}-1)+1\right)^{1/\beta}$.\\

We now turn to the case $\gamma>1$. It is readily checked that
$$
m_{\beta,\gamma,1,\eta}(\theta) \le V_1^* \ell_{\beta- \gamma, \eta}(\theta), 
$$
where
$$
V_1^* = \frac{1}{K^*_{\beta,\gamma,1,\eta}} (\eta-1)^{\gamma-1}\frac{\eta^{\beta - \gamma +1} -1} {\beta-\gamma +1}. 
$$
Letting
$$
\varphi_1^*(y) = \left(\frac{y-1}{\eta-1}\right)^{\gamma-1}
$$
leads to the following algorithm.\\

\noindent \textbf{Algorithm M1.} Simulation from $m_{\beta,\gamma,1,\eta}$ with $\gamma>1$.\\
\textbf{Step 1.} Independently simulate $U\sim U(0,1)$ and $Y\sim \ell_{\beta- \gamma, \eta}$.\\
\textbf{Step 2.} If $U_2\le \varphi_1^*(Y)$ return $Y$, otherwise go back to step 1.\\

On a given iteration, the probability of acceptance is $1/V^*_1$. From \eqref{eq: K* asymp} it follows that $1/V^*_1\to1/\gamma$ as $\eta\downarrow1$. Thus, when $\eta$ is close to $1$ this method works better for smaller values of $\gamma>1$.

\begin{remark}
Alternatively, we can note that for $\gamma>1$ and $\beta\ne0$
$$
m_{\beta,\gamma,1,\eta}(\theta) \le V_{1,1}^* \ell_{\beta- 1, \eta}(\theta),
$$
where
$$
V_{1,1}^* = \frac{1}{K^*_{\beta,\gamma,1,\eta}} \frac{\eta^{\beta} -1} {\beta}.
$$
This can be used to develop another rejection sampling method. This method may work better than Algorithm M1 for some choices of the parameters. However, \eqref{eq: K* asymp} implies that $1/V_{1,1}^* \to0$ as $\eta\downarrow1$. As such it will not work well for the situation of interest.
\end{remark}

For the remaining methods we only consider the case where $\gamma\in\{1,2,3,\dots\}$ is an integer. In this case the binomial theorem gives
\begin{equation}
m_{\beta,\gamma,1,\eta}(\theta) =\frac{1}{K^*_{\beta,\gamma,1,\eta}} \sum_{k=0}^{\gamma-1}{\gamma-1\choose k}(-1)^{k}\theta^{-1-k+\beta},\qquad 1<y<\eta .
\label{eq:mix:ig}
\end{equation}
Integrating from $1$ to $y\in(1,\eta)$
we get the \cdf
\begin{eqnarray}\label{eq: cdf IGa}
M_{\beta,\gamma,1,\eta}(y) =\frac{1}{K^*_{\beta,\gamma,1,\eta}} \sum_{k=0}^{\gamma-1}{\gamma-1\choose k}(-1)^{k}\frac{y^{\beta-k}-1}{\beta-k},
\end{eqnarray}
where, in the case $k=\beta$, we interpret $\frac{y^{-k+\beta}-1}{\beta-k}$ by its limiting value of $\ln y$. Let $M_{\beta,\gamma,1,\eta}^{-1}$ be the inverse function of $M_{\beta,\gamma,1,\eta}$. This can be calculated numerically, which leads to the following algorithm.\\

\noindent \textbf{Algorithm M2.} Simulation from $m_{\beta,\gamma,1,\eta}$ with $\gamma\in\{1,2,3,\dots\}$.\\
\textbf{Step 1.} Simulate $U\sim U(0,1)$.\\
\textbf{Step 2.} Return $M_{\beta,\gamma,1,\eta}^{-1}(U)$.\\

When $\gamma=1$ the inverse function has a simple form and this algorithm reduces to Algorithm M0. Our last algorithm is based on the methodology in \cite{BdM1971}. The idea is to use the positive terms in \eqref{eq:mix:ig} to obtain the bound
\begin{eqnarray*}
m_{\beta,\gamma,1,\eta}(\theta) &\le& \frac{1}{K^*_{\beta,\gamma,1,\eta}} \sum_{k=0}^{\lfloor(\gamma-1)/2\rfloor}{\gamma-1\choose 2k}\theta^{-1-2k+\beta}\\ 
&=&V^*_{2}\sum_{k=0}^{\lfloor(\gamma-1)/2\rfloor}p_{\beta,\gamma,\eta}(k)\,\ell_{\beta-2k-1,\eta}(\theta)=: V_2^*\tilde{m}_{\beta,\gamma,1,\eta}(\theta),\quad 1<\theta<\eta,
\end{eqnarray*}
where
\begin{equation*}
V_2^* = \frac{\sum_{k=0}^{\lfloor(\gamma-1)/2\rfloor}H^*_{2k,\beta,\gamma,\eta}}{K^*_{\beta,\gamma,1,\eta}},
\end{equation*}
\begin{equation*}
p_{\beta,\gamma,\eta}(k)=\frac{H^*_{2k,\beta,\gamma,\eta}}{\sum_{k=0}^{\lfloor(\gamma-1)/2\rfloor}H^*_{2k,\beta,\gamma,1,\eta}}, \ \  \ k=0, \dots,  \lfloor(\gamma-1)/2\rfloor,
\end{equation*}
and
$$
H^*_{k,\beta,\gamma,\eta} = {\gamma-1\choose k}\frac{\eta^{\beta - k}-1}{\beta -k}, \ \ \ k=0, \dots,  \gamma-1.
$$
Using the binomial theorem, it can be checked that $K^*_{\beta,\gamma,1,\eta}=\sum_{k=0}^{\gamma-1}(-1)^kH^*_{k,\beta,\gamma,\eta}$, which guarantees that $V_2^*\ge1$. Clearly, $\sum_{k=1}^{ \lfloor(\gamma-1)/2\rfloor}p_{\beta,\gamma,\eta}(k)=1$ and thus $p_{\beta,\gamma,\eta}$ is a valid \pmf. It follows that $\tilde{m}_{\beta,\gamma,1,\eta}$ is a mixture distribution and we can simulate from it as follows.\\

\noindent  \textbf{Algorithm BD.} Simulation from $\tilde{m}_{\beta,\gamma,1,\eta}$.\\
\noindent\textbf{Step 1.}  Simulate $S \sim p_{\beta,\gamma,\eta}$.\\
\noindent\textbf{Step 2.}  Simulate $Y \sim \ell_{\beta-2S-1,\eta}$ and return $Y$.\\

Now, letting
$$
\varphi_2^*(y) = \frac{(\theta-1)^{\gamma-1}\theta^{1-\gamma}}{\sum_{k=0}^{\lfloor(\gamma-1)/2\rfloor}{\gamma-1\choose 2k}\theta^{-2k}}
$$
leads to the following algorithm.\\

\noindent \textbf{Algorithm M3.} Simulation from $m_{\beta,\gamma,1,\eta}$ with $\gamma\in\{2,3,\dots\}$.\\
\textbf{Step 1.} Independently simulate $U\sim U(0,1)$ and $Y\sim \tilde{m}_{\beta,\gamma,1,\eta}$.\\
\textbf{Step 2.} If $U\le \varphi_2^*(Y)$ return $Y$, otherwise go back to step 1.\\

On a given iteration, the probability of acceptance is $1/V^*_2$. From l'H\^opital's Rule and  \eqref{eq: K* asymp} it follows that $1/V^*_2\to0$ as $\eta\downarrow1$. Thus, this method does not work well when $\eta$ is close to $1$. However, it may work well in other cases.

\subsection{Incomplete Beta Gamma Mixture Distribution}\label{subsec:new:dist}

We now introduce a distribution, which is important for the simulation of OUTS processes. To the best of our knowledge this distribution has not been studied previously. It has a \pdf\ of the form
\begin{eqnarray*}
f^\sharp_{\beta,\gamma,p,\eta}(v) = \frac{1}{C_{\beta,\gamma,p,\eta}} v^{p\gamma-\beta-1} \int_1^{\eta} \theta^{p\gamma-\beta-1} e^{-v^p\theta^p} \int_{1/\theta}^1 (1-u^p)^{\gamma-1} u^{-1-\beta} \rd u \rd \theta, \ \ v>0,
\end{eqnarray*}
where $p>0$, $\gamma>0$, and $\beta\in(-\infty, p\gamma)$ are parameters and $C_{\beta,\gamma,p,\eta}$ is a normalizing constant. We call this an incomplete beta gamma mixture (IBGM) distribution and we denote it by $\IBGM(\beta,\gamma,p,\eta)$. It is readily seen that
\begin{eqnarray*}
C_{\beta,\gamma,p,\eta} &=& p^{-1}\Gamma(\gamma-\beta/p) C^*_{\beta,\gamma,p,\eta},
\end{eqnarray*}
where
\begin{eqnarray*}
C^*_{\beta,\gamma,p,\eta} = \int_1^\eta \theta^{-1}\int_{1/\theta}^1 (1-u^p)^{\gamma-1} u^{-1-\beta}\rd u\rd \theta
= \int_{1/\eta}^1 \ln\left( \eta u\right)(1-u^p)^{\gamma-1} u^{-1-\beta}\rd u.
\end{eqnarray*}
It can be checked that
$$
f^\sharp_{\beta,\gamma,p,\eta}(v) =\int_1^\eta g_{p\gamma-\beta,p,\theta^p}(v)m^\sharp_{\beta,\gamma,p,\eta}(\theta)\rd \theta
$$
where 
\begin{eqnarray*}
m^\sharp_{\beta,\gamma,p,\eta}(\theta) &=&  \frac{1}{C^*_{\beta,\gamma,p,\eta} } \theta^{-1} \int_{1/\theta}^1 (1-u^p)^{\gamma-1} u^{-1-\beta} \rd u\\ 
 &=&  \frac{1}{pC^*_{\beta,\gamma,p,\eta} } \theta^{-1} \int_{1/\theta^p}^1 (1-u)^{\gamma-1} u^{-1-\beta/p} \rd u, \ \ 1<\theta<\eta.
\end{eqnarray*}
Thus, this is a GGSM with mixing density $m^\sharp_{\beta,\gamma,p,\eta}$. The presence of the incomplete beta function in the mixing density gives the distribution its name. It is easily checked that if $X\sim m^\sharp_{\beta/p,\gamma,1,\eta^p}$ then $X^{1/p}\sim m^\sharp_{\beta,\gamma,p,\eta}$.

We only focus on the case where $\gamma\in\{1,2,3,\dots\}$ as the other values are not relevant for simulating OUTS processes. In this case, the binomial theorem gives
\begin{eqnarray}\label{eq: m sharp expansion}
m^\sharp_{\beta,\gamma,p,\eta}(\theta) =  \frac{1}{C^*_{\beta,\gamma,p,\eta} } \sum_{k=0}^{\gamma-1} {\gamma-1\choose k} (-1)^k \theta^{-1}\frac{1-\theta^{\beta-pk}}{pk-\beta}, \quad 1<\theta<\eta,
\end{eqnarray}
where we replace $\frac{1-\theta^{\beta-pk}}{pk-\beta}$ by its limiting value of $\ln\theta$ when $\beta=pk$. Integrating shows that the \cdf\ is
\begin{eqnarray}\label{eq: m sharp cdf expansion}
M^\sharp_{\beta,\gamma,p,\eta}(y) =  \frac{1}{C^*_{\beta,\gamma,p,\eta} } \sum_{k=0}^{\gamma-1} {\gamma-1\choose k} (-1)^{k}\frac{(pk-\beta)\ln y-1+y^{\beta-pk}}{(pk-\beta)^2}, \quad 1<y<\eta,
\end{eqnarray}
where we replace $\frac{(pk-\beta)\log y-1+y^{\beta-pk}}{(pk-\beta)^2}$ by its limiting value of $(\ln y)^2/2$ when $\beta=pk$. Similarly, we can check that
$$
C^*_{\beta,\gamma,p,\eta} =  \sum_{k=0}^{\gamma-1} {\gamma-1\choose k} (-1)^{k}
\frac{(pk-\beta)\log \eta-1+\eta^{\beta-pk}}{(pk-\beta)^2}.
$$

\begin{prop}\label{prop: moments of IBGM}
If $W\sim \IBGM(\beta,\gamma,p,\eta)$ with $\gamma\in\{1,2,3,\dots\}$, then for any $\xi>\beta-\gamma p$ we have
$$
\rE[W^\xi] = \frac{\Gamma\left(\gamma+(\xi-\beta)/p\right)}{\Gamma\left(\gamma-\beta/p\right) C^*_{\beta,\gamma,p,\eta}}
\sum_{k=0}^{\gamma-1} {\gamma-1\choose k} \frac{ (-1)^{k}}{pk-\beta}\left(\frac{1-\eta^{-\xi}}{\xi} + \frac{1-\eta^{\beta-pk-\xi} }{\beta-pk-\xi}\right),
$$
where we replace $\frac{1}{pk-\beta}\left(\frac{1-\eta^{-\xi}}{\xi} + \frac{1-\eta^{\beta-pk-\xi} }{\beta-pk-\xi}\right)$ by $\frac{1-\eta^{-\xi}(\xi\ln\eta +1)}{\xi^2}$ if $\beta=pk$.
\end{prop}

We can use Algorithms GGSM1 and GGSM2 to simulate from the IBGM distribution. In this case Algorithm GGSM1 reduces to the following.\\

\noindent \textbf{Algorithm IBGM1.} Simulation from $\IBGM(\beta,\gamma,p,\eta)$.\\
\textbf{Step 1.} Independently simulate $Y\sim \Ga(\gamma-\beta/p,1)$ and $Z\sim m^\sharp_{\beta/p,\gamma,1,\eta^p}$.\\
\textbf{Step 2.} Return $\left(Y/Z\right)^{1/p}$.\\

To specialize Algorithm GGSM2 let $\varphi_2(u) = p^{-1}e^u\varphi_{2,n}(u)/\varphi_{2,d}(u)$, where 
\begin{eqnarray*}
\varphi_{2,n}(u) &=&  \sum_{k=0}^{\gamma-1} {\gamma-1\choose k} \frac{(-1)^{k}}{pk- \beta}u^{\beta/p-\gamma}
\int_u^{u \eta^p} e^{-\theta}\left(\theta^{\gamma-\beta/p-1}-\theta^{\gamma-k-1}u^{k-\beta/p}\right)\rd \theta
\end{eqnarray*}
and
\begin{eqnarray*}
\varphi_{2,d}(u) &=&  \sum_{k=0}^{\gamma-1} {\gamma-1\choose k} \frac{(-1)^{k}}{pk - \beta}\left(\frac{\eta^{p\gamma-\beta} - 1}{p\gamma-\beta}- 
\frac{\eta^{p(\gamma-k)} - 1}{p(\gamma-k)}\right).
\end{eqnarray*}
In the above, if  $k=\beta/p$, we replace the summand in $\varphi_{2,n}$ by
\begin{eqnarray*}
p^{-1}{\gamma-1\choose k} (-1)^{k}
\int_1^{\eta^p} e^{-u\theta}\theta^{\gamma-\beta/p-1}\ln \theta \rd \theta
\end{eqnarray*}
and the summand in $\varphi_{2,d}$ by
\begin{eqnarray*}
 {\gamma-1\choose k} (-1)^{k}\frac{\eta^{p\gamma-\beta}((p\gamma-\beta)\ln \eta - 1)+1}{(p\gamma-\beta)^2}.
\end{eqnarray*}
We note that for $k\ne \beta/p$ the integral in $\varphi_{2,n}$ can be written in terms of incomplete gamma functions. With this notation we can specialize Algorithm GGSM2 as follows.\\

\noindent \textbf{Algorithm IBGM2.} Simulation from $\IBGM(\beta,\gamma,p,\eta)$.\\
\textbf{Step 1.} Independently simulate $U\sim U(0,1)$ and $Y\sim \Ga(\gamma-\beta/p,1)$.\\
\textbf{Step 2.} If $U\le \varphi_2(Y)$ return $Y^{1/p}$, otherwise go back to step 1.\\
 
In this case, the probability of rejection is $1/V_2$, where
$$
V_2 = \frac{ \int_1^{\eta} \theta^{p\gamma-\beta-1} \int_{1/\theta}^1 (1-u^p)^{\gamma-1} u^{-1-\beta} \rd u \rd \theta}{ \int_1^\eta \theta^{-1}\int_{1/\theta}^1 (1-u^p)^{\gamma-1} u^{-1-\beta}\rd u\rd \theta}.
$$
Applying L'H\^opital's rule and Leibniz Rule shows that $V_2\to1$ as $\eta\downarrow1$.

We now turn to the problem of simulation from $m^\sharp_{\beta,\gamma,p,\eta}$. We begin with the important case when $\gamma=1$. In this case, $m^\sharp_{\beta,\gamma,p,\eta}$ does not depend on the parameter $p$. When $\gamma=1$ and $\beta=0$ we have 
$$
m^\sharp_{0,1,p,\eta}(\theta) = \frac{2 \log \theta}{\theta(\log \eta)^2}, \quad 1<\theta<\eta.
$$
It is easily checked that we can use the following algorithm in this case.\\

 \noindent \textbf{Algorithm M$^\sharp$0.} Simulation from $m^\sharp_{0,1,p,\eta}$.\\
\textbf{Step 1.} Simulate $U\sim U(0,1)$ and set $X= \eta^{\sqrt U}$.\\
\textbf{Step 2.} Return $X$.\\

When $\gamma=1$ and $\beta\ne0$ we get 
$$
m^\sharp_{\beta,1,p,\eta}(\theta) =  \frac{\beta}{\eta^{\beta}-\beta\log\eta-1} \frac{\theta^{\beta}-1}{\theta}, \quad 1<\theta<\eta.
$$
Simulation from this distribution was studied in \cite{QDZ21} and \cite{cs20_3}.  We follow the approach given in \cite{cs20_3}. The idea is to first observe that $X\sim m^\sharp_{\beta,1,p,\eta}$ can be represented as $X\eqd \eta^W$, where the \pdf\ of $W$ is
$$
f_W(w)=\frac{\beta\log\eta}{\eta^\beta - \beta\log\eta -1}\left(\eta^{\beta w} - 1\right), \quad 0\le w \le 1.
$$
Since $f_W$ is monotone and convex in $[0, 1]$, simulation can be done in a fast and efficient way by using the decomposition method illustrated in Section 4.3 of  \cite{Devroye86}.  
Let $0=w_0<w_1<\cdots<w_{L-1}<w_L=1$ for some positive integer $L$ and define a partition of $(0, 1]$ given by the disjoint intervals $\mathcal{I}_{\ell}=(w_{\ell-1},
w_{\ell}],\; \ell=1, \dots, L$. From the definition of a convex function, it follows that for each $\ell=1,2,\dots,L$
$$
f_W(w) \le \frac{f_W(w_{\ell}) - f_W(w_{\ell-1})}{w_{\ell} -
w_{\ell-1}}\,(w - w_{\ell-1}) + f_W(w_{\ell-1})=: g_{\ell}(w), \ \ w\in \mathcal{I}_{\ell}.
$$
Now let
\begin{align*}
    & \bar{g}_{\ell}(w) = \frac{g_{\ell}(w)}{q_{\ell}},\quad q_{\ell} =  \int_{\mathcal{I}_{\ell}} g_{\ell}(w)\rd w, \qquad
p_{\ell} = \frac{q_{\ell}}{V_L}, \qquad V_L = \sum_{\ell=1}^L q_{\ell}
\end{align*}
and note that each $\bar q_\ell$ is a \pdf\ and that
\begin{eqnarray}\label{eq: p}
p(\ell) = p_\ell, \ \ \ \ell=1,2,\dots, L
\end{eqnarray}
is a \pmf. With this notation we have
\begin{eqnarray}\label{eq: bound for f_W}
f_W(w) \le V_L\,\bar{g}_L(w),
\end{eqnarray}
where 
$$
\bar g_L(w) = \frac{1}{V_L}\sum_{\ell=1}^L g_{\ell}(w)\mathds{1}_{\mathcal{I}_{\ell}}(w) = \sum_{\ell=1}^L p_{\ell}\,\bar{g}_{\ell}(w)\mathds{1}_{\mathcal{I}_{\ell}}(w)
$$
is a \pdf. Using \eqref{eq: bound for f_W} we can set up a rejection sampling method for simulation from $f_W$. However, as observed in \cite{cs20_3}, the probability of acceptance can be made arbitrarily close to one when the intervals are of equal length and $L$ is large enough. Thus, in this case, we can skip the rejection sampling step and just use the approximation $f_W \approx\bar{g}_L$. Simulation results in \cite{cs20_3}, suggest that this approximation is very fast and works very well. It is almost exact. This leads to the following approximate simulation method for $m^\sharp_{\beta,1,p,\eta}$ with $\beta\ne0$.\\

\noindent \textbf{Algorithm M$^\sharp$-CS.} Approximate simulation from $m^\sharp_{\beta,1,p,\eta}$ with $\beta\ne0$.\\
\textbf{Step 1.} Simulate $\ell\sim p$, where $p$ is the \pmf\ in \eqref{eq: p}.\\
\textbf{Step 2.} Simulate $W_\ell\sim \bar g_\ell$.\\
\textbf{Step 3.} Return $X=\eta^{W_\ell}$.\\

In the above, simulation from $\bar g_\ell$ is straightforward as the \pdf\ is a linear function and we can use the inverse transform method, as its \cdf\ has a simple form. In practice we used the the \emph{random.triangular} routine in the \emph{numpy} package for \emph{Python}. We now turn to the case $\gamma>1$. A simple brute force approach is to numerically invert the formula for $M^\sharp_{\beta,\gamma,p,\eta}$ as given in \eqref{eq: m sharp cdf expansion}. Denoting this inverse function by $(M^\sharp_{\beta,\gamma,p,\eta})^{-1}$, leads to the following algorithm.\\
 
 \noindent \textbf{Algorithm M$^\sharp$1.} Simulation from $m^\sharp_{\beta,\gamma,p,\eta}$ with $\gamma\in\{1,2,3,\dots\}$.\\
\textbf{Step 1.} Simulate $U\sim U(0,1)$.\\
\textbf{Step 2.} Return $(M^\sharp_{\beta,\gamma,p,\eta})^{-1}(U)$.\\
 
 We also develop a rejection sampling algorithm, which follows from the fact that for $\gamma>1$
 $$
 m^\sharp_{\beta,\gamma,p,\eta}(\theta) \le \frac{C^*_{\beta,1,p,\eta}}{C^*_{\beta,\gamma,p,\eta}}m^\sharp_{\beta,1,p,\eta}(\theta).
 $$
Let 
$$
\varphi^\sharp_{1}(u) = \frac{\beta}{u^\beta-1}\sum_{k=0}^{\gamma-1} {\gamma-1\choose k} (-1)^k \frac{1-u^{\beta-pk}}{pk-\beta}
$$
where we interpret $\frac{\beta}{u^\beta-1}$ as $1/\ln u$ when $\beta=0$ and $\frac{1-u^{\beta-pk}}{pk-\beta}$ as $\ln u$ when $\beta=pk$. \\

\noindent \textbf{Algorithm M$^\sharp$2.} Simulation from $m^\sharp_{\beta,\gamma,p,\eta}$ with $\gamma\in\{2,3,\dots\}$.\\
\textbf{Step 1.} Independently simulate $U\sim U(0,1)$ and $Y\sim m^\sharp_{\beta,1,p,\eta}$.\\
\textbf{Step 2.} If $U\le \varphi^\sharp_1(Y)$ return $Y$, otherwise go back to step 1.\\

In this case the probability of acceptance on any given iteration is given by $\frac{C^*_{\beta,\gamma,p,\eta}}{C^*_{\beta,1,p,\eta}}$. Applying L'H\^opital's rule and Leibniz Rule shows that this approaches $0$ as $\eta\downarrow1$. Nevertheless, in simulations we found that this methods  works well for choices of $\eta$ that are not too close to $1$.

We end this section by noting that one can derive another rejection sampling algorithm by taking only the positive terms of \eqref{eq: m sharp expansion} as was done for the IGa law. Unfortunately, in this case, simulation from the law of the normalized positive sum is not straightforward as it requires an additional rejection sampling step. For this reason we do not to consider this approach here.

\subsection{Difference Generalized Gamma Distribution}\label{sec: DGGa}

Let $F$ be the \cdf\ of some distribution with support contained in $[0,\infty)$ and consider the function
$$
f(x) = \frac{F(\eta x)-F(x)}{x\log \eta} = \frac{1}{x\log \eta}\int_{(x,x\eta]} \rd F(v), \ \ x>0
$$
for some $\eta>0$. Is is readily checked that this is a \pdf. Such \pdf's arise in the study of the transition laws of OU processes, where the BDLP is compound Poisson, see \cite{Zhang:Sheng:Deng:2011}. We are interested in the case where $F$ is the \cdf\ of the $\GGa(\gamma,p,1)$ distribution. In this case, the \pdf\ becomes
\begin{eqnarray*}
h_{\gamma,p,\eta}(x) = \frac{1}{x\log \eta} \int_x^{x \eta} g_{\gamma,p,1} (u) \rd u = \frac{1}{\log \eta} \int_1^{ \eta} g_{\gamma,p,1} (\theta x) \rd \theta.
\end{eqnarray*}
We call this the Difference Generalized Gamma Distribution and denote it by $\DGGa(\gamma,p,\eta)$. It is readily checked that
\begin{eqnarray*}
h_{\gamma,p,\eta}(x) =  \int_1^{ \eta} g_{\gamma,p,\theta^p} ( x)m_{0,1,1,\eta}(\theta) \rd \theta,
\end{eqnarray*}
where
$$
m_{0,1,1,\eta}(\theta) = \frac{1}{\log\eta} \theta^{-1} \ \ 1<\theta<\eta.
$$
is a special case of the \pdf\ given in \eqref{eq: m for IGa}. We can use Algorithm M0 to simulate from $m_{0,1,1,\eta}$ and we can combine this with Algorithm GGSM1 to simulate from $\DGGa(\gamma,p,\eta)$.

\section{$p$-Tempered $\alpha$-Stable Distributions}\label{sec: TS}

Fix $\alpha\in(-\infty,2)$ and $p>0$. A $p$-tempered $\alpha$-stable ($\ts$) distribution $\mu$ on $\mathbb R^d$ has a characteristic function of the form $\hat\mu(z) = e^{C_\mu(z)}$, where
$$
c_\mu(z) = i\langle b,z\rangle + \int_{\mathbb R^d}\int_0^\infty \left(e^{it\langle x,z\rangle}-1-it\langle x,z\rangle 1_{[\alpha\ge1]}\right) t^{-1-\alpha}e^{-t^p} \rd t R(\rd x),
$$
$b\in\mathbb R^d$, and $R$ is a finite Borel measure on $\mathbb R^d$ satisfying $R(\{0\})=0$ and
\begin{eqnarray*}
\int_{|x|>2} |x|^\alpha R(\rd x) <\infty && \mbox{ if }\alpha\in(0,2)\setminus\{1\}\\
\int_{|x|>2} |x|\log|x| R(\rd x) <\infty && \mbox{ if }\alpha=1\\
\int_{|x|>2} \log|x| R(\rd x) <\infty && \mbox{ if }\alpha=0.
\end{eqnarray*}
No additional assumptions on $R$ are needed when $\alpha<0$. We denote this distribution by $\ts(R,b)$. We call $b$ the shift and $R$ the Rosi\'nski measure after the author of \cite{Rosinski07}. One can consider extensions to certain cases where $R$ is not a finite measure (see \cite{Grabchak16}), but we will not do so here. The class of $\ts$ distributions with $p=1$ and $\alpha\in(0,2)$ was introduce in \cite{Rosinski07} and the class with $p=2$ and $\alpha\in[0,2)$ was introduced in \cite{Bianchi:etal:2011}. The general class was introduced in \cite{Grabchak12}.

Every $p$-tempered $\alpha$-stable distribution is infinitely divisible and the L\'evy measure of $\ts(R,b)$ is given by
\begin{eqnarray}\label{eq: Levy measure of TS}
M(B) = \int_{\mathbb R^d}\int_0^\infty 1_B(xt) t^{-1-\alpha}e^{-t^p}\rd t R(\rd x), \ \ \ B\in\mathfrak B(\mathbb R^d)
\end{eqnarray}
where $\mathfrak B(\mathbb R^d)$ denotes the Borel sets on $\mathbb R^d$. Formulas for the cumulants of $\ts$ distributions are given in Theorem 2.16 of \cite{Grabchak16}. For simplicity we only recall the formulas in the one-dimensional ($d=1$) case. In this case for distribution $\ts(R,b)$ if $\int_{|x|>1}|x|^k R(\rd x)<\infty$, then the $k$th cumulant exists and is given for $\alpha\in(-\infty,1)$ by
\begin{eqnarray}\label{eq: cum TS}
\mathrm{c}_k = p^{-1}\Gamma\left(\frac{k-\alpha}{p}\right)\int_{\mathbb R}x^k R(\rd x) + 1_{[k=1]}b.
\end{eqnarray}
If $\alpha\in[1,2)$, then this formula still holds for $k\ge2$, but for $k=1$ it is given by $\mathrm{c}_1=b$. For more on $p$-tempered $\alpha$-stable distributions and their associate L\'evy processes see \cite{Grabchak16} and the references therein. 

While we present our results for general Rosi\'nski measures $R$, we are especially interest in the class of so-called $p$-rapidly decreasing tempered stable ($p$-RDTS) distributions, see \cite{Grabchak21a}. These correspond to the case where the dimension $d=1$ and the Rosi\'nski measure is of the form $R(\rd x) = c\beta^\alpha \delta_{1/\beta}(\rd x)$ for some $c,\beta>0$. In this case, after a change of variables, we get
$$
c_\mu(z) = ib z + c \int_0^\infty \left(e^{itz}-1-itz 1_{[\alpha\ge1]}\right) t^{-1-\alpha}e^{-(\beta t)^p} \rd t.
$$
If one understands these distributions, then one can easily extend to the bilateral case, where $R(\rd x) = c_-\beta_-^\alpha \delta_{-1/\beta_-}(\rd x)+c_+\beta_+^\alpha \delta_{1/\beta_+}(\rd x)$ for some $c_-,c_+,\beta_-,\beta_+>0$.  When $p=1$ these are sometimes called classical tempered stable (CTS) distributions. 
\section{OU Processes}\label{sec: OU processes}

Let $L=\{L_t:t\ge0\}$ be a L\'evy process on $\mathbb R^d$. Fix $\lambda>0$ and define a process $Y=\{Y_t:t\ge0\}$ as the strong solution of the stochastic differential equation (\sde)
            \begin{equation*}
              \rd Y_t =  -\lambda Y_t\rd t + \rd L_t, \qquad Y_0=Y\ \ a.s. 
    \end{equation*}
This process can be written as
\begin{equation*}
Y_t = Y_0\,e^{-\lambda t} + \int_0^t e^{-\lambda\,(t-s)}\rd L_s.
\end{equation*}
We say that $Y$ is an OU process with parameter $\lambda$ and that $L$ is the BDLP.  We refer to the distribution of $L_1$ as the BDLP distribution. Every OU process is a Markov process and, so long as $\rE[\log(|L_1|\vee e)]<\infty$, the process has a stationary (also sometimes called a limiting) distribution. An OU process whose stationary distribution is $\ts$ is called a TSOU process and an OU process whose BDLP distribution is $\ts$ is called an OUTS process. 

In the remainder of this section we study the transition laws of both TSOU and OUTS processes. We begin by giving formulas for their cumulants. For simplicity we focus on the one-dimensional ($d=1$) case. In \cite{SabinoCufaro20} a simple formula relating the cumulants of the transition law of an OU process and those of the stationary law are provided. Specifically, it is shown there that, if the stationary law has a finite $k$th cumulant, then so does the transition law and, in this case, the $k$th cumulant of the conditional distribution of $Y_{s+t}$ given $Y_s=y$ is given by
$$
\hat c_{k,t} = ye^{-\lambda t}1_{[k=1]}+(1-e^{-k\lambda t})c_{k},
$$
where $c_k$ is the $k$th cumulant of the stationary law. Thus, when $d=1$ and the stationary law is $\ts(R,b)$, if $\int_{|x|>1}|x|^k R(\rd x)<\infty$, then
\begin{eqnarray}\label{eq: cum TSOU}
\hat c_{k,t} = \left(ye^{-\lambda t} + (1-e^{-k\lambda t})b\right)1_{[k=1]} + (1-e^{-k\lambda t})p^{-1}\Gamma\left(\frac{k-\alpha}{p}\right)\int_{\mathbb R}x^k R(\rd x).
\end{eqnarray}
The one exception to this formula is that when $k=1$ and $\alpha\in[1,2)$ we have $\hat c_{1,t} = \left(ye^{-\lambda t} + (1-e^{-k\lambda t})b\right)$. We now turn to the cumulants of the transition law of an OUTS process. In this case, combining Proposition 3.12 in \cite{ContTankov2004} with Lemma 17.1 in \cite{Sato} shows that, so long as the $k$th cumulant of the BDLP distribution exists, the $k$th cumulant of the conditional distribution of $Y_{s+t}$ given $Y_s=y$ exists and is given by
$$
\check c_{k,t} = ye^{-\lambda t}1_{[k=1]}+\frac{1-e^{-k\lambda t}}{k\lambda}c_{k},
$$
where $c_k$ is the $k$th cumulant of the BDLP distribution. Thus, when $d=1$ and the BDLP distribution is $\ts(R, b)$, if  $\int_{|x|>1}|x|^k R(\rd x)<\infty$, then
\begin{eqnarray}\label{eq: cum OUTS}
\check c_{k,t} = \left(ye^{-\lambda t}+ \frac{1-e^{-k\lambda t}}{k\lambda} b\right)1_{[k=1]} +  \frac{1-e^{-k\lambda t}}{k\lambda p} \Gamma\left(\frac{k-\alpha}{p}\right)\int_{\mathbb R}x^k R(\rd x).
\end{eqnarray}
We must again modify this formula when $k=1$ and $\alpha\in[1,2)$. In this case we have $\check c_{1,t} = \left(ye^{-\lambda t} + (1-e^{-k\lambda t})(k\lambda)^{-1}b\right)$. 

We now give explicit representations for the transition laws of both TSOU and OUTS processes and discuss simulation. These are given in $d$-dimensions. We begin with TSOU processes. Since only selfdecomposable distributions can serve is stationary distributions of OU processes, we only consider the case when $\alpha\in[0,2)$ as $\ts$ distributions are not selfdecomposable when $\alpha<0$, see Proposition 3.14 in \cite{Grabchak16}. The following result is given in \cite{Grabchak21}.

\begin{thm}\label{thrm: main}
Let $Y=\{Y_t:t\ge0\}$ be a TSOU process with parameter $\lambda>0$ and stationary distribution $\ts(R,b)$ with $p>0$, $\alpha\in[0,2)$, and $R\ne0$. Set $\gamma=1+\lfloor\alpha/p\rfloor$. If $t>0$, then, given $Y_s=y$, we have
\begin{eqnarray}\label{eq: trans RV:ptsou}
Y_{s+t} \eqd e^{-\lambda t} y + (1-e^{-\lambda t})b - \sum_{n=0}^{\gamma-1}b_n + X_0 + e^{-\lambda t} \sum_{n=1}^{\gamma-1}X_n + \sum_{j=n}^{N} V_n W_n, 
\end{eqnarray}
where $b_0,\dots,b_{\gamma-1}\in\mathbb R^d$ are constants and $N, X_0, X_1,\dots,X_{\gamma-1}$, $V_1, V_2, \dots$, $W_1, W_2, \dots$ are independent random variables with:\\
1. $X_0\sim \ts(R_0,0)$ with $R_0(\rd x) = (1-e^{-\alpha \lambda t})R(\rd x)$,\\
2. if $\gamma\ge2$ then $X_n \sim \mathrm{TS}^p_{\alpha-np}(R_n,0)$ with $R_n(\rd x) = \frac{1}{n!}(1-e^{-p \lambda t})^n R(\rd x)$ for $n=1,2,\dots,(\gamma-1)$,\\
3. $V_1, V_2, \dots\siid R^1$, where $R^1(\rd x)=R(\rd x)/R(\mathbb R^{d})$,\\
4. $W_1, W_2, \dots\siid\IGa(\alpha,\gamma,p,e^{p\lambda t})$,\\
5. $N$ has a Poisson distribution with mean $e^{-\alpha \lambda t}R(\mathbb R^d)K_{\alpha,\gamma,p,e^{p\lambda t}}$,\\
6.
$$
b_0 = \left\{\begin{array}{ll} e^{-\alpha\lambda t} \int_{\mathbb R^{d}} x  R(\rd x) K_{(\alpha-1),\gamma,p,e^{p\lambda t}} 
& \alpha\in[1,2)\\
0 & \alpha\in[0,1)
\end{array}\right.,
$$
and if $\gamma\ge2$ then for $n=1,2,\dots,(\gamma-1)$
$$
b_n =  \left\{\begin{array}{ll} 
e^{-\lambda t} \int_{\mathbb R^{d}}x R_n(\rd x)  p^{-1}\Gamma\left(\frac{1-\alpha+np}{p}\right)   & 1\le \alpha<1+ np\\
0 & \mbox{otherwise}
\end{array}\right..
$$
\end{thm}

Note that when $\alpha=0$ we have $\gamma=1$, $b_0=0$, and $X_0=0$ with probability one, thus the transition law is essentially just compound Poisson. Note further, that the $\IGa$ distribution needed in the theorem has parameter $\eta=e^{p\lambda t}$. When simulating a TSOU process on a finite grid, one typically takes a small time step $t>0$. Thus one often uses a value of $\eta$ that is close to $1$. Next, we turn to OUTS processes. In this case we can allow for any $\alpha\in(-\infty,2)$. To the best of our knowledge the transition law has not been studied previously in this case, except for CTS and closely related distributions and only in the one-dimensional case with $\alpha\in[0,1)$. 

\begin{thm}\label{thrm: main bdlp}
Let $Y=\{Y_t:t\ge0\}$ be an OUTS process with parameter $\lambda>0$ and BDLP distribution $\ts(\lambda R,\lambda b)$ with $p>0$, $\alpha\in(-\infty,2)$, and $R\ne0$. If $\alpha\in(0,2)$ set $\gamma=1+\lfloor\alpha/p\rfloor$, otherwise set $\gamma=1$. If $t>0$, then, given $Y_s=y$, we have
\begin{eqnarray}\label{eq: trans RV: oupts}
Y_{s+t} \eqd e^{-\lambda t} y + (1-e^{-\lambda t})b - \sum_{n=0}^{\gamma-1}b^*_n + e^{-\lambda t} \sum_{n=0}^{\gamma-1}X_n + \sum_{n=1}^{N} V_n W_n, 
\end{eqnarray}
where $b_0,\dots,b_{\gamma-1}\in\mathbb R^d$ are constants and $N, X_0, X_1,\dots,X_{\gamma-1}$, $V_1, V_2, \dots$, $W_1, W_2, \dots$ are independent random variables with:\\
1. $X_0\sim \ts(R^*_0,0)$ with $R^*_0(\rd x) = \frac{e^{\alpha\lambda t}-1}{\alpha}R(\rd x)$,\\
2. if $\gamma\ge2$ then $X_n \sim \mathrm{TS}^p_{\alpha-np}(R^*_n,0)$ with $R^*_n(\rd x) = \kappa_{\lambda,t,n}R(\rd x)$ and $\kappa_{\lambda,t,n} =\int^1_{e^{-\lambda t}}\frac{(1-u^p)^n}{n!} u^{-1-\alpha} \rd u $,  for $n=1,2,\dots,(\gamma-1)$,\\
3. $V_1, V_2, \dots\siid R^1$, where $R^1(\rd x)=R(\rd x)/R(\mathbb R^{d})$,\\
4. $W_1, W_2, \dots\siid \IBGM(\alpha,\gamma,p,e^{\lambda t})$,\\
5. $N$ has a Poisson distribution with mean $\frac{pC_{\alpha,\gamma,p,e^{\lambda t}} R(\mathbb R^d)}{(\gamma-1)!}$,\\
6.
$$
b^*_0 = \left\{\begin{array}{ll}  \frac{pC_{\alpha,\gamma,p,e^{\lambda t}} R(\mathbb R^d)}{(\gamma-1)!} \rE[V_1]\rE[W_1]
& \alpha\in[1,2)\\
0 & \alpha<1
\end{array}\right.,
$$
and if $\gamma\ge2$ then for $n=1,2,\dots,(\gamma-1)$
$$
b^*_n =  \left\{\begin{array}{ll} 
e^{-\lambda t} \int_{\mathbb R^{d}}x R^*_n(\rd x)  p^{-1}\Gamma\left(\frac{1-\alpha+np}{p}\right)   & 1\le \alpha<1+ np\\
0 & \mbox{otherwise}
\end{array}\right..
$$
\end{thm}

In the theorem and its proof, when $\alpha=0$, we interpret $\frac{e^{\alpha\lambda t}-1}{\alpha}$ by its limiting value of $t\lambda$. We can, of course, state the theorem for the case where the  BDLP distribution is $\ts(R, b)$ instead of $\ts(\lambda R,\lambda b)$. However, the formulas would be a bit more complicated and we do not do so here. 

Note that Theorem \ref{thrm: main bdlp} holds even if the OUTS process does not have a stationary distribution. A stationary distribution exists if and only if
$$
\int_{|x|>2} \log |x| M(\rd x)<\infty,
$$
where $M$ is the L\'evy measure of $\ts(R,b)$. A simple sufficient condition is  
$$
\int_{|x|>2}|x|^\epsilon R(\rd x)<\infty \mbox{ for some } \epsilon>0.
$$
Under our assumptions, this always holds for $\alpha\in(0,2)$, see \cite{Grabchak16}. 

While Theorem \ref{thrm: main bdlp} holds for any $\alpha\in(-\infty,2)$, when $\alpha<0$ we can get a significantly simpler representation as, in this case, $\ts$ distributions are simply compound Poisson (with drift). In the one-dimensional case, a general representation of the transition law of an OU process with a compound Poisson BDLP is given in \cite{Zhang:Sheng:Deng:2011}. Although we cannot use those results directly as we are in $d$-dimensions, our results are related to the ones in that paper.

\begin{thm}\label{thrm: main bdlp alpha neg}
Let $Y=\{Y_t:t\ge0\}$ be an OUTS process with parameter $\lambda>0$ and BDLP distribution $\ts(\lambda R,\lambda b)$ with $p>0$, $\alpha\in(-\infty,0)$, and $R\ne0$. If $t>0$, then, given $Y_s=y$, we have
\begin{eqnarray}\label{eq: trans RV^:oupts:neg}
Y_{s+t} \eqd e^{-\lambda t} y + (1-e^{-\lambda t})b +\sum_{n=1}^{N} V_n W_n, 
\end{eqnarray}
where $N$, $V_1, V_2, \dots$, $W_1, W_2, \dots$ are independent random variables with:\\
1. $V_1, V_2, \dots\siid R^1$, where $R^1(\rd x)=R(\rd x)/R(\mathbb R^{d})$,\\
2. $W_1, W_2, \dots\siid \DGGa(|\alpha|,p,e^{\lambda t})$,\\
3. $N$ has a Poisson distribution with mean $p^{-1}\lambda t \Gamma(|\alpha|/p)R(\mathbb R^d)$.
\end{thm}

Our main goal in studying the transition laws is to use them to simulate the corresponding TSOU or OUTS process on a finite grid. To do this, we need a way to simulate from the transition law, or equivalently from the various components of this law. We have already discussed the simulation of $\IGa$, $\IBGM$, and $\DGGa$ distributions in Section \ref{sec: GGSM}. There is no one approach for simulating from $R^1$ as it can be, essentially, any probability measure on $\mathbb R^d$. However, when simulating specifically TSOU processes, there is a way to avoid simulating from $R^1$. In this case one can directly simulate the product $V_iW_i$, where $V_i\sim R^1$ and $W_i$ has an $\IGa$ distribution, see \cite{Grabchak21}.

The remaining components of the transition law are $\mathrm{TS}^p_{\alpha-np}$ for $n=0,1,\dots,\gamma-1$. There are several approaches for simulating from these distributions. First, one can use the inverse transform method, which requires one to numerically invert the \cdf. While this method can work well, the fact that there is no closed formula for the \cdf s of $\ts$ distributions makes this method impractical in many cases. Second, under mild assumptions, one can use the rejection sampling approach of \cite{Grabchak19}. However, this method requires one to numerically calculate \pdf s, which may also be computationally intensive. A third approach is to use an approximate method based on truncating an infinite series representation. A number of such representations appear in the literature, see \cite{Rosinski07}, \cite{Rosinski:Sinclair:2010}, \cite{Imai:Kawai:2011}, or \cite{Bianchi:etal:2011}. We note that several of the methods discussed here are easier to implement in the univariate case. An approach for extending univariate simulation methods of $\ts$ random variables to the multivariate case is given in \cite{Xia:Grabchak:2022}. Finally, we note that numerical methods for simulation and the evaluation of \pdf s and \cdf s of certain classes of symmetric $\ts$ distribution can be found in the SymTS package \cite{Grabchak:Cao:2017} for the statistical software R.

\begin{remark}\label{remark: pRDST sim}
We are particularly interested in the class of $p$-RDTS distributions, which correspond to the case where the dimension $d=1$ and $R(\rd x) =c\beta^\alpha\delta_{1/\beta}(\rd x)$ for some $c,\beta>0$. In this case simulation of the various components of the transition law is fairly simple. First, we have $R^1(\rd x) = \delta_{1/\beta}(\rd x)$ and thus if $V\sim R^1$ then $V=1/\beta$ with probability $1$. Second, a simple method for simulating from $\ts(R,b)$ is given in \cite{Grabchak21a} for the case $\alpha<1$ and $p>1$. Finally, when $p=1$, this class reduces to the class of CTS distributions. Exact simulation methods for CTS distributions are well known and can be found in, e.g., \cite{Devroye:2009}, \cite{Hofert:2011}, \cite{Kawai:Masuda:2011}, and the references therein.
\end{remark}

\section{Numerical Experiments}\label{sec:num:exp}

In this section we illustrate and compare the performance and effectiveness of the simulation algorithms discussed in this paper. All simulations were conducted using \emph{Python} with a $64$-bit Intel Core i7-7500U CPU @270-290 GHz, 8GB. We first investigate the performance of the simulation methods for the IGa distribution as described in Section\myref{sec:iga:dist} and then the simulation methods for the IBGM distribution presented in Section\myref{subsec:new:dist}. Finally, we focus on the generation of TSOU and OUTS processes on a finite grid. To ensure that we are simulating from the correct distributions, we compare the empirical moments to the true moments. For simplicity, for the OU processes we compare the cumulants instead of the moments. To see how close the empirical values are to the true values, we consider the relative errors as given by
\begin{equation*}
   \text{err}\; \% = \frac{\text{true value - estimated value}}{\text{true value}} \times 100\%.
\end{equation*}

\subsection{Results for IGa}

In this section we compare the performance of four methods for simulating from an IGa distribution, which are discussed in Section \ref{sec:iga:dist}. Three of them are new and use Algorithm IGa1 in conjunction with an algorithm for simulating from $m_{\beta,\gamma,p,\eta}$. 
We denote these by \emph{ARGS}, \emph{Inverse}, and \emph{ARBD} and for $\gamma\ge2$ they use Algorithms M1, M2, and M3, respectively. When $\gamma=1$, Algorithms M1 and M2 are no longer meaningful and Algorithm M2 reduces to Algorithm M0. For this reason, when $\gamma=1$ we use Algorithm M0 for all three methods. The fourth method uses Algorithm IGa2 and is denoted \emph{ARG}. It was introduced in \cite{Grabchak21}. 

Simulation using the \emph{Inverse} method when $\gamma\ge 2$ depends on the numerical inversion of the \cdf\ given in \refeqq{eq: cdf IGa}. This, in turn, depends on an initial guess which can, of course, affect the final computation time. Instead of blindly taking the middle term $(\eta^{1/p}+1)/2$, we choose the initial guess equal to the random variate drawn from the corresponding distribution with $\gamma=1$. In other words, we start with the value returned by Algorithm M0. Calculating this value is fast and its impact on the overall computation time is negligible.

 As discussed in Section\myref{sec:iga:dist}, without loss of generality we take $p=1$. For the other parameters, we take $\beta=0.9$,  $\eta\in\{1.1, 2\}$, and $\gamma\in\{1, 2, 3, 4, 5, 10\}$. The choice of the $\eta$'s stems from the fact that when simulating a TSOU process on a finite grid one often needs $\eta>1$ close to $1$. Table\myref{tab:iga:times} shows the computation times for the four methods. We use the method \emph{ARGS} as the baseline and for it all values are given in seconds, while the values for the other methods are given as multiplicative factors with respect to it. We can see that the new \emph{ARGS} method performed the fastest, while the new \emph{ARBD} method performed the slowest. Further, the new methods \emph{ARGS} and \emph{Inverse} performed significantly faster than the \emph{ARG} method of \cite{Grabchak21}. 

To ensure that the methods are simulating from the correct distributions, Table\myref{tab:iga:moments} shows the comparison between the first four true moments computed in \eqref{eq: moments of IGa} and the empirical moments estimated based on $R=5\times 10^4$ simulated values. In the interest of space, we only present the results for $p=1$, $\beta=0.9$, and $\eta=2$. In all cases the \emph{err \%} is small suggesting that all methods are simulating from the correct distributions.

\begin{table}[ht!]
    \centering\footnotesize
        \resizebox{\textwidth}{!}{
        \begin{tabular}{|c|c|cc|cc|cc|cc|cc|cc|}
				             \hline
				& & \multicolumn{2}{c|}{$\gamma=1$} & \multicolumn{2}{c|}{$\gamma=2$} & \multicolumn{2}{c|}{$\gamma=3$} & \multicolumn{2}{c|}{$\gamma=4$} & \multicolumn{2}{c|}{$\gamma=5$} & \multicolumn{2}{c|}{$\gamma=10$}\\
										\hline
Method & $R$ & $\eta_1$ & $\eta_2$ & $\eta_1$ & $\eta_2$ & $\eta_1$ & $\eta_2$ & $\eta_1$ & $\eta_2$ & $\eta_1$ & $\eta_2$ & $\eta_1$ & $\eta_2$\\										 \hline\hline 			

\multirow{ 4}{*}{ARGS (sec)} &
$1000$ & $0.00020$ & $0.00040$ & $0.0030$ & $0.0080$ & $0.0049$ & $0.0100$ & $0.0080$ & $0.0170$ & $0.0120$ & $0.0269$ & $0.0289$ & $0.6902$\\
& $10000$ & $0.00125$ & $0.00289$ & $0.0498$ & $0.0549$ & $0.0678$ & $0.0858$ & $0.0817$ & $0.1685$ & $0.1067$ & $0.2822$ & $0.2604$ & $6.8407$\\
& $20000$ & $0.00247$ & $0.00598$ & $0.1054$ & $0.0858$ & $0.1197$ & $0.1685$ & $0.1576$ & $0.3311$ & $0.2055$ & $0.5586$ & $0.5037$ & $13.559$\\
& $50000$ & $0.00669$ & $0.00740$ & $0.2042$ & $0.2105$ & $0.2942$ & $0.4488$ & $0.4019$ & $0.7660$ & $0.4957$ & $1.4192$ & $1.2825$ & $38.205$\\
                    \hline\hline
\multirow{ 4}{*}{Inverse} &
$1000$ & $1$ & $1$ & $2.3$ & $1.6$ & $1.9$ & $2.5$ & $1.9$ & $2.2$ & $4.6$ & $2.0$ & $1.7$ & $1.9$\\
& $10000$ & $1$ & $1$ & $2.9$ & $1.4$ & $1.9$ & $2.5$ & $3.1$ & $3.1$ & $4.2$ & $4.0$ & $3.0$ & $1.6$\\
& $20000$ & $1$ & $1$ & $2.4$ & $2.8$ & $1.7$ & $5.2$ & $3.9$ & $3.1$ & $3.5$ & $4.2$ & $2.5$ & $1.9$\\
& $50000$ & $1$ & $1$ & $4.0$ & $2.2$ & $3.1$ & $4.9$ & $3.0$ & $4.2$ & $4.7$ & $4.3$ & $2.7$ & $3.5$\\
                    \hline\hline
\multirow{ 4}{*}{ARG} &
$1000$ & $4.5$ & $6.4$ & $17$ & $6.8$ & $12$ & $9.2$ & $7.5$ & $7.6$ & $5.9$ & $6.4$ & $3.0$ & $4.7$\\
& $10000$ & $5.7$ & $6.0$ & $6.4$ & $10$ & $8.4$ & $10$ & $7.5$ & $7.6$ & $6.6$ & $8.3$ & $3.4$ & $4.8$\\
& $20000$ & $6.1$ & $6.5$ & $4.6$ & $12$ & $9.4$ & $11$ & $7.8$ & $8.8$ & $6.5$ & $11$ & $3.5$ & $3.9$\\
& $50000$ & $12$ & $11$ & $8.5$ & $12$ & $9.6$ & $15$ & $7.7$ & $22$ & $6.8$ & $20$ & $3.4$ & $3.5$\\
                    \hline\hline
\multirow{ 4}{*}{ARBD} &
$1000$ & $1$ & $1$ & $21$ & $18$ & $24$ & $26$ & $19$ & $20$ & $15$ & $15$ & $6.8$ & $7.2$\\
& $10000$ & $1$ & $1$ & $11$ & $25$ & $18$ & $21$ & $18$ & $21$ & $20$ & $17$ & $8.8$ & $8.5$\\
& $20000$ & $1$ & $1$ & $12$ & $28$ & $21$ & $23$ & $20$ & $22$ & $33$ & $28$ & $11$ & $11$\\
& $50000$ & $1$ & $1$ & $13$ & $30$ & $22$ & $22$ & $44$ & $54$ & $48$ & $47$ & $12$ & $10$\\
                    \hline\hline	
        \end{tabular}
        }
    \scriptsize
    \caption{\footnotesize{Results for IGa: computation times. Here we take $(\eta_1, \eta_2)=(1.1, 2)$ and $(p, \beta)=(1, 0.9)$. For \emph{ARGS} the values are in seconds, otherwise they are the multiplicative factors with respect to the \emph{ARGS} method}}\label{tab:iga:times}
\end{table}

\begin{table}[ht!]
    \centering\scriptsize
        \resizebox{\textwidth}{!}{
        \begin{tabular}{*{9}{|c|ccccc|ccccc}}
				             \hline
				 & \multicolumn{5}{c|}{$m_1$} & \multicolumn{5}{c|}{$m_2$} \\
                    \hline
$\gamma$ & True & ARGS & Inverse  & ARG  & ARBD  
& True & ARGS & Inverse  & ARG  & ARBD \\
										\hline
$1$ & $0.070$ & $1.0\%$ & $1.0\%$ & $-0.5\%$ & $1.0\%$
& $0.055$ & $1.3\%$ & $1.3\%$ & $1.9\%$ & $1.3\%$ \\
$2$ & $0.694$ & $0.5\%$ & $0.1\%$ & $-0.3\%$ & $-2.8\%$ 
& $0.946$ & $0.8\%$ & $-0.5\%$ & $-0.5\%$ & $-5.1\%$ \\
$3$ & $1.261$ & $0.3\%$ & $0.0\%$ & $0.2\%$ & $0.0\%$ 
& $2.395$ & $0.7\%$ & $-0.3\%$ & $0.1\%$ & $-0.2\%$\\ 
$4$ & $1.804$ & $0.1\%$ & $0.2\%$ & $0.0\%$ & $0.8\%$ 
& $4.372$ & $0.4\%$ & $0.4\%$ & $0.3\%$ & $1.6\%$ \\
$5$ & $2.336$ & $0.0\%$ & $-0.2\%$ & $0.2\%$ & $1.1\%$ 
& $6.866$ & $0.0\%$ & $-0.5\%$ & $0.5\%$ & $2.3\%$ \\
$10$ & $4.916$ & $0.2\%$ & $0.6\%$ & $-0.5\%$ & $-0.5\%$ 
& $26.96$ & $0.3\%$ & $1.5\%$ & $-0.6\%$ & $-3.7\%$\\
                    \hline
				 & \multicolumn{5}{c|}{$m_3$} & \multicolumn{5}{c|}{$m_4$} \\
                    \hline
$\gamma$ & True & ARGS & Inverse  & ARG  & ARBD
& True & ARGS & Inverse  & ARG  & ARBD \\
										\hline									
$1$ & $0.088$ & $2.9\%$ & $2.9\%$ & $4.0\%$ & $2.9\%$ 
& $0.212$ & $-3.5\%$ & $-3.5\%$ & $-2.9\%$ & $-3.5\%$ \\
$2$ & $1.962$ & $0.8\%$ & $-1.5\%$ & $0.1\%$ & $-7.4\%$ 
& $5.546$ & $-0.3\%$ & $-2.5\%$ & $1.6\%$ & $3.4\%$ \\
$3$ & $6.152$ & $1.0\%$ & $-0.9\%$ & $-0.3\%$ & $0.1\%$  
& $20.13$ & $1.1\%$ & $-1.5\%$ & $-1.2\%$ & $1.3\%$ \\
$4$ & $13.40$ & $1.0\%$ & $0.5\%$ & $0.9\%$ & $2.5\%$
& $50.06$ & $1.9\%$ & $0.8\%$ & $1.9\%$ & $3.0\%$ \\
$5$ & $24.47$ & $-0.1\%$ & $-0.8\%$ & $0.8\%$ & $3.6\%$
& $103.0$ & $-0.3\%$ & $-1.0\%$ & $1.1\%$ & $5.3\%$ \\
$10$ & $163.3$ & $0.4\%$ & $1.7\%$ & $-1.8\%$ & $3.4\%$  
& $1085$ & $0.4\%$ & $5.1\%$ & $-4.0\%$ & $6.1\%$ \\
  \hline				
        \end{tabular}
        }
    \scriptsize
  \caption{\footnotesize{Results for IGa: moment comparison. Here we take $(p, \beta, \eta)=(1, 0.9, 2)$ and evaluate empirical moments based on $R=5\times 10^4$ simulated values. Column \emph{True} gives the true values of the moments, while the other columns give the err \%.}}\label{tab:iga:moments}
\end{table}

\subsection{Results for IBGM}\label{sec:num:exp:new:dist}

In this section we compare the performance of three methods for simulating from an IBGM distribution, which are presented in Section\myref{subsec:new:dist}. The first method, denoted \emph{Inverse} combines Algorithm IBGM1 with Algorithm M$^\sharp$1. In this case, we always chose our initial guess for the numerical inversion to be the midpoint, $(1+\eta)/2$. The second, denoted \emph{ARGS} combines Algorithm IBGM1 with Algorithm M$^\sharp$-CS  when $\gamma=1$ and with Algorithm M$^\sharp$2 when $\gamma\ge2$. The third, denoted \emph{GGSM}, uses Algorithm IBGM2. When implementing Algorithm M$^\sharp$2, we use Algorithm M$^\sharp$-CS in the first step to generate an observation from $m^\sharp_{\beta,1,p,\eta}$. In all cases, when we use Algorithm M$^\sharp$-CS we take $L=2000$ equally spaced intervals. This algorithm was introduced in\mycite{cs20_3} and, while it is an approximate algorithm, it works very well and is almost exact. 

Table\myref{tab:new:dist:times} presents the computation times of the different methods for several choices of the parameters. Here, we take \emph{Inverse} as the baseline. For it all values are given in seconds, while the values for the other methods are given as multiplicative factors with respect to it. When $\gamma=1$, \emph{ARGS}, which uses Algorithm M$^\sharp$-CS, performed the fastest. However, when $\gamma\ge2$ there was a dichotomy. In this case \emph{GGSM} is always faster than \emph{ARGS} for $\eta=1.1$ and slower than \emph{ARGS} for $\eta=2$. This is likely related to the asymptotic results (as $\eta\downarrow1$) for the probability of acceptance in Algorithms IBGM2 and $M^\sharp2$, which are given in Section\myref{subsec:new:dist}. Method \emph{Inverse} tends to work better for larger values for $\gamma$. To summarize, when $\gamma=1$ it is better to use \emph{ARGS}, when $\gamma$ is large it is better to use \emph{Inverse}, and when $\gamma\ge2$ is not too big, the situation depends on the value of $\eta$. In the context of the simulation of an OUTS process on a finite grid, a larger value of $\eta$ corresponds to a grid of time points with larger time-steps, whereas $\eta$ approaches $1$ as the grid gets finer. Thus, in this case, the selection of the fastest approach depending on the granularity of the grid.

Table\myref{tab:new:distribution:moments} gives the comparison between the true and the empirical moments based on $5\times 10^4$ simulated values. We see that all of the methods seem to be simulating from the correct distributions. This is especially important to note in the case of \emph{ARGS} as this method is only approximate.  We see that it works well and that the err \% is no worse than it is for the other methods.

\begin{table}[ht!]
    \centering\footnotesize
        \resizebox{\textwidth}{!}{
        \begin{tabular}{|c|c|cc|cc|cc|cc|cc|cc|}
				             \hline
				& & \multicolumn{2}{c|}{$\gamma=1$} & \multicolumn{2}{c|}{$\gamma=2$} & \multicolumn{2}{c|}{$\gamma=3$} & \multicolumn{2}{c|}{$\gamma=4$} & \multicolumn{2}{c|}{$\gamma=5$} & \multicolumn{2}{c|}{$\gamma=10$}\\
										\hline
Method & $R$ & $\eta_1$ & $\eta_2$ & $\eta_1$ & $\eta_2$ & $\eta_1$ & $\eta_2$ & $\eta_1$ & $\eta_2$ & $\eta_1$ & $\eta_2$ & $\eta_1$ & $\eta_2$\\										 \hline\hline 	
\multirow{ 4}{*}{Inverse (sec)} &
$1000$ & $0.4853$ & $1.0034$ & $0.5923$ & $0.5855$ & $1.2915$ & $0.6171$ & $0.8489$ & $1.2302$ & $0.0385$ & $0.4398$ & $0.0204$ & $0.7247$\\
& $10000$ & $2.2132$ & $2.6426$ & $3.5021$ & $1.5841$ & $8.9314$ & $2.6201$ & $2.0610$ & $3.4214$ & $0.2153$ & $3.6592$ & $0.1602$ & $4.5543$\\
& $20000$ & $3.9827$ & $3.2911$ & $8.8375$ & $5.7192$ & $15.896$ & $6.4364$ & $3.6689$ & $8.9187$ & $0.4276$ & $12.247$ & $0.3198$ & $12.514$\\
& $50000$ & $12.636$ & $9.1716$ & $27.139$ & $17.987$ & $45.205$ & $23.394$ & $8.253$ & $36.293$ & $1.0157$ & $48.394$ & $0.792$ & $21.124$\\

										\hline\hline
\multirow{ 4}{*}{ARGS} &
$1000$ & $0.0082$ & $0.0030$ & $1.6$ & $1.5$ & $0.5$ & $1.2$ & $0.8$ & $0.7$ & $18$ & $2.1$ & $386$ & $1.8$\\
& $10000$ & $0.0050$ & $0.0034$ & $1.8$ & $4.6$ & $0.7$ & $3.6$ & $3.2$ & $2.5$ & $33$ & $2.5$ & $478$ & $2.9$\\
& $20000$ & $0.0048$ & $0.0061$ & $1.4$ & $3.2$ & $0.8$ & $2.9$ & $3.7$ & $1.9$ & $33$ & $1.5$ & $455$ & $2.1$\\
& $50000$ & $0.0050$ & $0.0030$ & $1.1$ & $2.5$ & $0.7$ & $2.1$ & $4.1$ & $1.4$ & $35$ & $1.0$ & $449$ & $3.1$\\
                    \hline\hline
\multirow{ 4}{*}{GGSM} &
$1000$ & $0.1665$ & $0.0885$ & $0.2$ & $3.0$ & $0.1$ & $8.4$ & $0.2$ & $1.5$ & $6.8$ & $6.1$ & $37$ & $219$\\
& $10000$ & $0.3880$ & $0.2597$ & $0.3$ & $18$ & $0.2$ & $17$ & $1.0$ & $15$ & $12$ & $7.0$ & $46$ & $331$\\
& $20000$ & $0.3176$ & $0.3285$ & $0.2$ & $9.8$ & $0.2$ & $14$ & $1.1$ & $7.1$ & $12$ & $3.9$ & $46$ & $241$\\
& $50000$ & $0.1822$ & $0.2909$ & $0.2$ & $7.2$ & $0.2$ & $11$ & $1.2$ & $16$ & $13$ & $13$ & $48$ & $351$\\
							      \hline\hline				
        \end{tabular}
        }
    \scriptsize
    \caption{\footnotesize{Results for IBGM: computation times. Here we take $(\eta_1, \eta_2)=(1.1, 2)$ and $(p, \beta)=(1, 0.9)$. For \emph{Inverse} the values are in seconds, otherwise they are the multiplicative factors with respect to the \emph{Inverse} method}}\label{tab:new:dist:times}
\end{table}

\begin{table}[ht!]
    \centering\footnotesize
        \resizebox{\textwidth}{!}{
        \begin{tabular}{*{9}{|c|cccc|cccc|}}
				             \hline
				 & \multicolumn{4}{c|}{$m_1$} & \multicolumn{4}{c|}{$m_2$} \\
                    \hline
$\gamma$ & True & Inverse  & ARGS  & GGSM  & True & Inverse  & ARG  & ARBD  \\
										\hline
$1$ & $0.0630$ & $0.2\%$ & $0.1\%$ & $0.3\%$ & $0.0450$ & $-0.2\%$ & $0.2\%$ & $1.6\%$\\
$2$ & $0.6566$ & $-0.1\%$ & $0.1\%$ & $0.4\%$ & $0.8393$ & $0.6\%$ & $0.4\%$ & $0.5\%$\\
$3$ & $1.2146$ & $-0.3\%$ & $-1.0\%$ & $-0.1\%$ & $2.2096$ & $0.5\%$ & $-1.8\%$ & $0.1\%$\\
$4$ & $1.7556$ & $0.1\%$ & $-2.3\%$ & $-0.3\%$ & $4.1220$ & $0.2\%$ & $-4.5\%$ & $-0.9\%$\\
$5$ & $2.2865$ & $0.1\%$ & $-3.0\%$ & $0.0\%$ & $6.5618$ & $-0.1\%$ & $-1.7\%$ & $0.0\%$\\
$10$ & $4.8739$ & $0.1\%$ & $-0.7\%$ & $-3.0\%$ & $26.470$ & $0.2\%$ & $-0.9\%$ & $-3.2\%$\\

                    \hline
				 & \multicolumn{4}{c|}{$m_3$} & \multicolumn{4}{c|}{$m_4$} \\
                    \hline
$\gamma$ & True & Inverse  & ARGS  & GGSM  & True & Inverse  & ARGS  & GGSM  \\
										\hline									
$1$ & $0.0630$ & $-2.5\%$ & $-2.5\%$ & $4.6\%$ & $0.1346$ & $4.7\%$ & $6.4\%$ & $-4.9\%$\\
$2$ & $1.6182$ & $-2.1\%$ & $1.3\%$ & $0.1\%$ & $4.2236$ & $-5.3\%$ & $2.8\%$ & $0.1\%$\\
$3$ & $5.4023$ & $0.5\%$ & $-2.0\%$ & $0.4\%$ & $16.724$ & $3.5\%$ & $-0.7\%$ & $0.8\%$\\
$4$ & $12.190$ & $-0.1\%$ & $-5.8\%$ & $-1.7\%$ & $43.721$ & $-1.0\%$ & $-5.1\%$ & $-3.2\%$\\
$5$ & $22.749$ & $-0.3\%$ & $-4.6\%$ & $-0.1\%$ & $92.820$ & $-0.6\%$ & $-2.7\%$ & $-0.3\%$\\
$10$ & $158.67$ & $0.2\%$ & $-0.7\%$ & $-3.0\%$ & $1041.7$ & $0.3\%$ & $1.6\%$ & $-0.7\%$\\
							      \hline				
        \end{tabular}
        }
    \scriptsize
    \caption{\footnotesize{Results for IBGM: moment comparison.  Here we take $(p, \beta, \eta)=(1, 0.9, 2)$ and evaluate empirical moments based on $R=5\times 10^4$ simulated values. Column \emph{True} gives the true values of the moments, while the other columns give the err \%.}}\label{tab:new:distribution:moments}
\end{table}

\subsection{Results for TSOU processes}\label{sub:num:oupts}

Theorem\myref{thrm: main} characterizes the transition laws of TSOU processes with $p>0$ and $\alpha\in[0, 2)$. This can be used to simulate such a process on a finite grid of times. 
In this section we illustrate this approach by performing a series of simulations. We focus on the important class of $p$-RDTS distributions, which correspond to the case where the dimension $d=1$ and the stationary distribution is $\ts(R,b)$, where $R(\rd x)=c\beta^{\alpha}\delta_{1/\beta}(\rd x)$ for some $c,\beta>0$. This means that $R^1(\rd x) = \delta_{1/\beta}(\rd x)$ and hence that each $V_i=1/\beta$ with probability $1$. For simplicity we take $b=0$ and for tractability, we take $\alpha<1$ and $p>1$. In this case $\gamma=1$, we can simulate $X_0$ using the method given in \cite{Grabchak21a}, and we can simulate from the required $\IGa$ distribution using the \emph{Inverse} method, which combines Algorithm IGa1 with Algorithm M0.

Figure\myref{fig:ptsou:trajectories} displays the sample trajectories of TSOU processes with several choices of the parameters. These were simulated using the time-step $t=1/365$ over an equally-spaced grid with $365$ points. It is well-known that $p$-RDTS distributions with $\alpha=0.5$ and $p=1$ reduce to the well-known class of inverse Gaussian distributions. In our simulations we take $\alpha=0.5$ and $p>1$.  Thus, these processes are generalizations of inverse Gaussian OU processes. We note that the transition laws of inverse Gaussian OU processes were studied in \cite{Zhang:Zhang:2008}.

Next, we check the correctness and of our algorithm. We simulate $10^5$ observations from the stationary law with a time step of $t=0.1$. A comparison of the true cumulants and the empirical cumulants is given in Table\myref{tab:pts:ou:30:360} for several choices of the parameters. The values of the true cumulants are evaluated using \eqref{eq: cum TSOU} with the appropriate choice of $R$ and the parameters. We see that for all cumulants the err \% is small, which suggests that the algorithm stemming from Theorem\myref{thrm: main} is simulating from the correct distribution. 

\begin{figure}
\centering
\includegraphics[width=16cm, height=6cm]{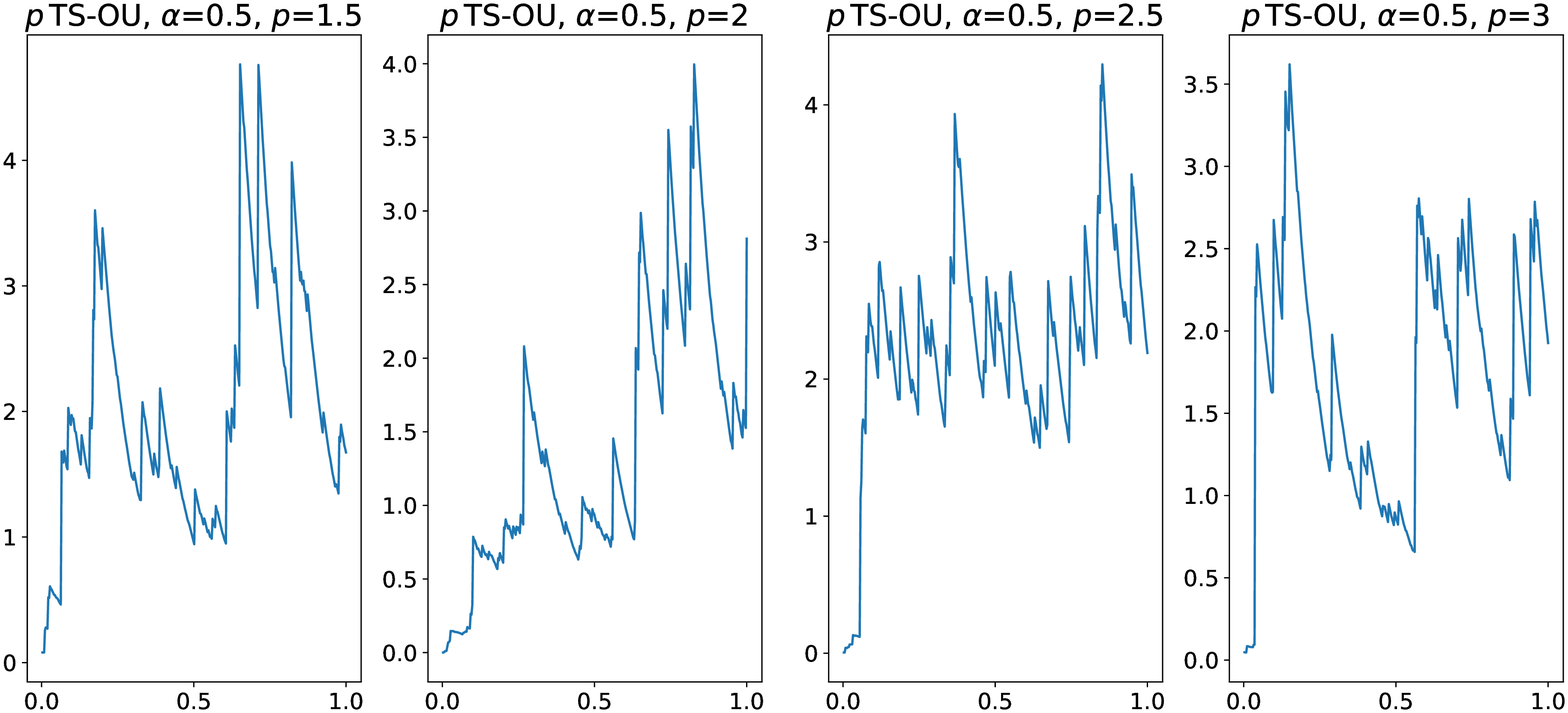}
\caption{Sample trajectories of TSOU processes with parameter $\lambda=10$, initial value $y_0=0$, and time step $t=1/365$. The stationary distribution is $\ts(R,0)$ with $R(\rd x) = c\beta^\alpha\delta_{1/\beta}(\rd x)$, where $c=1$, $\beta=1$, $\alpha=0.5$, and $p\in\{1.5,2,2.5,3\}$.}
\label{fig:ptsou:trajectories}
\end{figure}

\begin{table}[ht!]
    \centering\scriptsize
        \resizebox{\textwidth}{!}{
        \begin{tabular}{*{10}{|c|c||rr|rr|rr|rr}}
                    \hline
& &       \multicolumn{2}{c|}{$c_{X, 1}(0,t)$} & \multicolumn{2}{c|}{$c_{X, 2}(0,t)$} & \multicolumn{2}{c|}{$c_{X, 3}(0,t)$} & \multicolumn{2}{c|}{$c_{X, 4}(0,t)$} \\
                    \hline
                   $p$ &  $\alpha$ & true & err \% & true & err \% & true & err \% & true & err \% \\
                    \hline
\multirow{ 5}{*}{1.5} & 
$0.1$ & $0.628$ & $-0.61\%$ & $0.521$ & $-0.53\%$ & $0.617$ & $-0.34\%$ & $0.936$ & $1.25\%$\\
& $0.3$ & $0.800$ & $-0.03\%$ & $0.541$ & $-0.04\%$ & $0.590$ & $-1.07\%$ & $0.850$ & $-2.93\%$\\
& $0.5$ & $1.129$ & $0.06\%$ & $0.576$ & $-0.47\%$ & $0.572$ & $-2.52\%$ & $0.779$ & $-4.01\%$\\
& $0.7$ & $1.935$ & $-0.04\%$ & $0.632$ & $0.29\%$ & $0.562$ & $0.10\%$ & $0.721$ & $-1.58\%$\\
& $0.9$ & $6.104$ & $-0.71\%$ & $0.720$ & $-0.48\%$ & $0.562$ & $-1.74\%$ & $0.674$ & $-1.14\%$\\

\hline\hline
\multirow{ 5}{*}{2} & 
$0.1$ & $0.622$ & $-0.69\%$ & $0.446$ & $-0.76\%$ & $0.421$ & $-0.56\%$ & $0.481$ & $-0.01\%$\\
& $0.3$ & $0.805$ & $0.53\%$ & $0.481$ & $0.47\%$ & $0.423$ & $1.06\%$ & $0.464$ & $4.53\%$\\
& $0.5$ & $1.146$ & $0.31\%$ & $0.530$ & $0.28\%$ & $0.431$ & $-1.11\%$ & $0.451$ & $-4.04\%$\\
& $0.7$ & $1.966$ & $-0.04\%$ & $0.599$ & $0.15\%$ & $0.443$ & $-0.75\%$ & $0.442$ & $-4.39\%$\\
& $0.9$ & $6.154$ & $0.63\%$ & $0.699$ & $-0.78\%$ & $0.463$ & $-3.14\%$ & $0.436$ & $-4.21\%$\\

\hline\hline
\multirow{ 5}{*}{2.5} & 
$0.1$ & $0.625$ & $0.69\%$ & $0.419$ & $0.36\%$ & $0.353$ & $-0.65\%$ & $0.349$ & $-1.95\%$\\
& $0.3$ & $0.813$ & $0.35\%$ & $0.460$ & $0.56\%$ & $0.365$ & $0.32\%$ & $0.348$ & $-4.66\%$\\
& $0.5$ & $1.161$ & $-0.03\%$ & $0.515$ & $-0.98\%$ & $0.380$ & $-2.13\%$ & $0.348$ & $-3.77\%$\\
& $0.7$ & $1.988$ & $0.04\%$ & $0.590$ & $0.03\%$ & $0.400$ & $0.50\%$ & $0.351$ & $2.29\%$\\
& $0.9$ & $6.185$ & $-0.34\%$ & $0.696$ & $0.77\%$ & $0.427$ & $-2.16\%$ & $0.357$ & $3.67\%$\\

\hline\hline
\multirow{ 5}{*}{3} & 
$0.1$ & $0.630$ & $-0.36\%$ & $0.409$ & $0.24\%$ & $0.323$ & $0.97\%$ & $0.294$ & $-1.64\%$\\
& $0.3$ & $0.822$ & $0.25\%$ & $0.453$ & $0.03\%$ & $0.339$ & $-0.54\%$ & $0.298$ & $-3.60\%$\\
& $0.5$ & $1.173$ & $-0.22\%$ & $0.511$ & $-0.49\%$ & $0.358$ & $-1.28\%$ & $0.304$ & $-2.44\%$\\
& $0.7$ & $2.005$ & $-0.18\%$ & $0.589$ & $-0.85\%$ & $0.381$ & $-3.16\%$ & $0.311$ & $-3.13\%$\\
& $0.9$ & $6.206$ & $-0.62\%$ & $0.699$ & $-2.17\%$ & $0.411$ & $-2.23\%$ & $0.321$ & $-4.23\%$\\
                    \hline
        \end{tabular}
        }
    \scriptsize
    \caption{\footnotesize{Comparison of the first four true cumulants with their estimated values obtained from $10^5$ simulations from the transition law of a TSOU process with $\lambda=10$, initial value $y_0=0$, and time step $t=0.1$. The stationary distribution is $\ts(R,0)$ with $R(\rd x) = c\beta^\alpha\delta_{1/\beta}(\rd x)$, where $c=1$, $\beta=1$, and with several choices for $\alpha$ and $p$.
}}\label{tab:pts:ou:30:360}
\end{table}

\subsection{OUTS processes}\label{sub:num:cts}

We now turn to the simulation of OUTS process on a finite grid. We again focus on the case of $p$-RDTS distributions and for simplicity we assume that the shift $b=0$. Here, we are assuming that the dimension $d=1$ and that the BDLP distribution is $\ts(\lambda R,0)$, where $\lambda>0$ is the parameter of the OUTS process and $R(\rd x) = c\beta^\alpha\delta_{1/\beta}(\rd x)$ for some $c,\beta>0$. 

We begin with the case $\alpha\ge0$, for which the transition law is characterized in Theorem\myref{thrm: main bdlp}. For tractability, we again focus on the case $\alpha\in[0,1)$ and $p>1$. Here $\gamma=1$ and we can simulate the term $X_0$ using the approach given  in\mycite{Grabchak21a}. To simulate from the $\IBGM$ distribution we use the \emph{ARGS} method discussed in Section\myref{sec:num:exp:new:dist}. We again take $L=2000$ equally spaced intervals, which leads to an efficient approximate simulation method. Figure\myref{fig:ptsou:neg:trajectories} shows sample trajectories of OUTS processes for several choices of the parameters. In all cases we take $y_0=0$ as the initial value, a time step of $t=1/365$, $\alpha=0.5$, and $p>1$. With this choice for $\alpha$, we can think of the processes as extensions of OU process with inverse Gaussian BDLP distributions. Next, to check the correctness of the algorithm, we simulate $10^5$ observations from the transition law with a time step of $t=0.1$ and several values for the parameters. We evaluate the empirical cumulants and compare them to the true cumulants in Table\myref{tab:ou:pts:30:360}. We can see that err \% is small. Here the true cumulants are evaluated using \eqref{eq: cum OUTS}. 

When $\alpha<0$ the transition law is given in Theorem\myref{thrm: main bdlp alpha neg}. The formula is very simple and essentially boils down to simulating from the $\DGGa$ distribution, which is easily done using the approach described in Section \ref{sec: DGGa}. For several choices of the parameters, plots of the sample trajectories of these processes using a time step of $t=1/365$ are given in Table\myref{tab:ou:pts:negative:30:360} and a comparison of the empirical and true cumulants again using $10^5$ observations from the transition law is given in Figure\myref{fig:ptsou:neg:trajectories}. We can again see that err \% is small.

\begin{table}[ht!]
    \centering\scriptsize
        \resizebox{\textwidth}{!}{
        \begin{tabular}{*{10}{|c|c||rr|rr|rr|rr}}
                    \hline
& &       \multicolumn{2}{c|}{$c_{X, 1}(0,t)$} & \multicolumn{2}{c|}{$c_{X, 2}(0,t)$} & \multicolumn{2}{c|}{$c_{X, 3}(0,t)$} & \multicolumn{2}{c|}{$c_{X, 4}(0,t)$} \\
                    \hline
                   $p$ &  $\alpha$ & true  & err \% & true  & err \% & true  & err \% & true  & err \% \\
                    \hline
\multirow{ 5}{*}{1.5} &
$0.1$ & $0.063$ & $1.0\%$ & $0.026$ & $0.8\%$ & $0.021$ & $-1.3\%$ & $0.023$ & $-3.6\%$\\
& $0.3$ & $0.080$ & $0.1\%$ & $0.027$ & $1.1\%$ & $0.020$ & $2.6\%$ & $0.021$ & $3.8\%$\\
& $0.5$ & $0.113$ & $0.8\%$ & $0.029$ & $2.0\%$ & $0.019$ & $3.7\%$ & $0.020$ & $4.1\%$\\
& $0.7$ & $0.194$ & $0.1\%$ & $0.032$ & $0.5\%$ & $0.019$ & $0.6\%$ & $0.018$ & $0.3\%$\\
& $0.9$ & $0.610$ & $-0.3\%$ & $0.036$ & $-2.0\%$ & $0.019$ & $-4.0\%$ & $0.017$ & $-4.6\%$\\

\hline\hline
\multirow{ 5}{*}{2} & 
$0.1$ & $0.062$ & $1.0\%$ & $0.022$ & $2.2\%$ & $0.014$ & $3.7\%$ & $0.012$ & $3.8\%$\\
& $0.3$ & $0.081$ & $1.3\%$ & $0.024$ & $2.1\%$ & $0.014$ & $1.2\%$ & $0.012$ & $-3.8\%$\\
& $0.5$ & $0.115$ & $0.4\%$ & $0.027$ & $0.2\%$ & $0.014$ & $-0.6\%$ & $0.011$ & $-0.7\%$\\
& $0.7$ & $0.197$ & $-0.6\%$ & $0.030$ & $-1.2\%$ & $0.015$ & $0.3\%$ & $0.011$ & $4.2\%$\\
& $0.9$ & $0.615$ & $-0.2\%$ & $0.035$ & $-1.3\%$ & $0.015$ & $-1.2\%$ & $0.011$ & $0.8\%$\\

\hline\hline
\multirow{ 5}{*}{2.5} &
$0.1$ & $0.063$ & $1.4\%$ & $0.021$ & $1.7\%$ & $0.012$ & $1.5\%$ & $0.009$ & $0.3\%$\\
& $0.3$ & $0.081$ & $0.8\%$ & $0.023$ & $1.4\%$ & $0.012$ & $-0.4\%$ & $0.009$ & $-4.0\%$\\
& $0.5$ & $0.116$ & $0.3\%$ & $0.026$ & $0.9\%$ & $0.013$ & $1.1\%$ & $0.009$ & $1.6\%$\\
& $0.7$ & $0.199$ & $0.4\%$ & $0.030$ & $2.6\%$ & $0.013$ & $3.8\%$ & $0.009$ & $4.5\%$\\
& $0.9$ & $0.619$ & $-0.1\%$ & $0.035$ & $-0.2\%$ & $0.014$ & $-1.2\%$ & $0.009$ & $-2.2\%$\\

\hline\hline
\multirow{ 5}{*}{3} & 
$0.1$ & $0.063$ & $-0.7\%$ & $0.020$ & $-1.6\%$ & $0.011$ & $-2.6\%$ & $0.007$ & $-3.0\%$\\
& $0.3$ & $0.082$ & $-0.5\%$ & $0.023$ & $-1.5\%$ & $0.011$ & $-2.6\%$ & $0.007$ & $-5.3\%$\\
& $0.5$ & $0.117$ & $0.5\%$ & $0.026$ & $0.4\%$ & $0.012$ & $0.1\%$ & $0.008$ & $0.3\%$\\
& $0.7$ & $0.201$ & $-0.1\%$ & $0.030$ & $0.0\%$ & $0.013$ & $0.7\%$ & $0.008$ & $2.2\%$\\
& $0.9$ & $0.621$ & $0.0\%$ & $0.035$ & $0.8\%$ & $0.014$ & $0.9\%$ & $0.008$ & $0.3\%$\\

\hline
        \end{tabular}
        }
    \scriptsize    \caption{\footnotesize{Comparison of the first four true cumulants with their estimated values obtained from $10^5$ simulations from the transition law of an OUTS process with $\lambda=10$, initial value $y_0=0$, and time step $t=0.1$. The BDLP distribution is $\ts(\lambda R,0)$ with $R(\rd x) = c\beta^\alpha\delta_{1/\beta}(\rd x)$, where $c=0.1$, $\beta=1$, and with several choices for $\alpha>0$ and $p>1$.}}\label{tab:ou:pts:30:360}
\end{table}

\begin{figure}
\centering
\includegraphics[width=16cm, height=6cm]{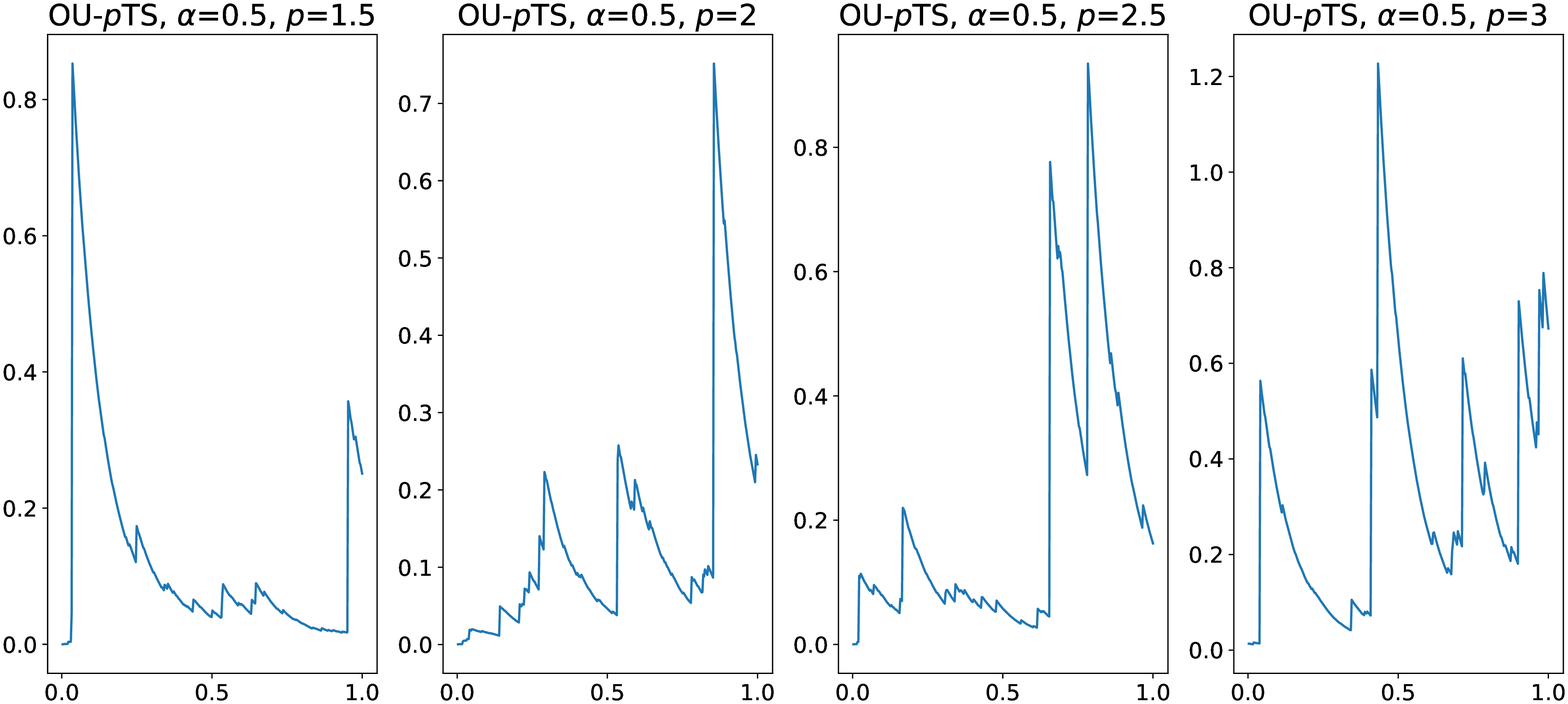}
\caption{Sample trajectories of OUTS processes with parameter $\lambda=10$, initial value $y_0=0$, and time step $t=1/365$. The BDLP distribution is $\ts(\lambda R,0)$ with $R(\rd x) = c\beta^\alpha\delta_{1/\beta}(\rd x)$, where $c=0.1$, $\beta=1$, $\alpha=0.5$, and $p\in\{1.5,2,2.5,3\}$.}
\label{fig:oupts:trajectories}
\end{figure}

\begin{table}[ht!]
    \centering\scriptsize
        \resizebox{\textwidth}{!}{
        \begin{tabular}{*{10}{|c|c||rr|rr|rr|rr}}
                    \hline
& &       \multicolumn{2}{c|}{$c_{X, 1}(0,t)$} & \multicolumn{2}{c|}{$c_{X, 2}(0,t)$} & \multicolumn{2}{c|}{$c_{X, 3}(0,t)$} & \multicolumn{2}{c|}{$c_{X, 4}(0,t)$} \\
                    \hline
                   $p$ &  $\alpha$ & true & err \% & true &  err \% & true& err \% & true &  err \% \\
                    \hline
\multirow{ 5}{*}{1.5} & 
$-0.1$ & $0.053$ & $1.17\%$ & $0.026$ & $1.73\%$ & $0.022$ & $1.89\%$ & $0.026$ & $1.32\%$\\
& $-0.3$ & $0.046$ & $-0.16\%$ & $0.026$ & $0.34\%$ & $0.023$ & $0.00\%$ & $0.029$ & $-3.25\%$\\
& $-0.5$ & $0.042$ & $1.28\%$ & $0.026$ & $1.41\%$ & $0.025$ & $2.43\%$ & $0.033$ & $5.27\%$\\
& $-0.7$ & $0.040$ & $-2.02\%$ & $0.027$ & $-1.41\%$ & $0.027$ & $-0.46\%$ & $0.037$ & $0.70\%$\\
& $-0.9$ & $0.038$ & $-0.13\%$ & $0.028$ & $-0.43\%$ & $0.030$ & $-0.16\%$ & $0.042$ & $0.87\%$\\
\hline\hline
\multirow{ 5}{*}{2} & 
$-0.1$ & $0.051$ & $0.97\%$ & $0.021$ & $2.48\%$ & $0.014$ & $5.99\%$ & $0.013$ & $6.12\%$\\
& $-0.3$ & $0.044$ & $0.01\%$ & $0.020$ & $0.26\%$ & $0.014$ & $2.02\%$ & $0.013$ & $6.27\%$\\
& $-0.5$ & $0.039$ & $-0.09\%$ & $0.020$ & $0.70\%$ & $0.015$ & $2.44\%$ & $0.014$ & $4.09\%$\\
& $-0.7$ & $0.035$ & $0.78\%$ & $0.019$ & $0.55\%$ & $0.015$ & $0.81\%$ & $0.015$ & $2.08\%$\\
& $-0.9$ & $0.033$ & $0.67\%$ & $0.019$ & $1.36\%$ & $0.016$ & $0.74\%$ & $0.016$ & $-1.95\%$\\
\hline\hline
\multirow{ 5}{*}{2.5} & 
$-0.1$ & $0.051$ & $-0.38\%$ & $0.019$ & $-0.27\%$ & $0.012$ & $1.50\%$ & $0.009$ & $6.01\%$\\
& $-0.3$ & $0.043$ & $1.55\%$ & $0.018$ & $4.21\%$ & $0.011$ & $7.91\%$ & $0.009$ & $6.40\%$\\
& $-0.5$ & $0.038$ & $-1.85\%$ & $0.017$ & $-3.46\%$ & $0.011$ & $-5.58\%$ & $0.009$ & $-4.02\%$\\
& $-0.7$ & $0.034$ & $-0.67\%$ & $0.017$ & $-0.71\%$ & $0.011$ & $-1.32\%$ & $0.009$ & $-2.60\%$\\
& $-0.9$ & $0.031$ & $-0.54\%$ & $0.016$ & $1.02\%$ & $0.011$ & $3.21\%$ & $0.010$ & $6.10\%$\\
\hline\hline
\multirow{ 5}{*}{3} & 
$-0.1$ & $0.051$ & $-0.09\%$ & $0.019$ & $0.62\%$ & $0.010$ & $1.74\%$ & $0.007$ & $4.35\%$\\
& $-0.3$ & $0.043$ & $0.13\%$ & $0.017$ & $0.56\%$ & $0.010$ & $1.97\%$ & $0.007$ & $5.30\%$\\
& $-0.5$ & $0.037$ & $0.10\%$ & $0.016$ & $-1.69\%$ & $0.010$ & $-4.80\%$ & $0.007$ & $-4.52\%$\\
& $-0.7$ & $0.033$ & $1.02\%$ & $0.015$ & $1.22\%$ & $0.010$ & $2.16\%$ & $0.007$ & $4.73\%$\\
& $-0.9$ & $0.030$ & $-0.19\%$ & $0.015$ & $-1.00\%$ & $0.009$ & $-2.36\%$ & $0.007$ & $-3.78\%$\\
										
\hline
        \end{tabular}
        }
    \scriptsize \caption{\footnotesize{Comparison of the first four true cumulants with their estimated values obtained from $10^5$ simulations from the transition law of an OUTS process with $\lambda=10$, initial value $y_0=0$, and time step $t=0.1$. The BDLP distribution is $\ts(\lambda R,0)$ with $R(\rd x) = c\beta^\alpha\delta_{1/\beta}(\rd x)$, where $c=0.1$, $\beta=1$, and with several choices for $\alpha<0$ and $p>1$.
 }}\label{tab:ou:pts:negative:30:360}
\end{table}

Comparison of the first four true cumulants with their estimated values obtained from $10^5$ simulations from the transition law of an OUTS process with $\lambda=10$, initial value $y_0=0$, and time step $t=0.1$. The stationary distribution is $\ts(R,0)$ with $R(\rd x) = c\beta^\alpha\delta_{1/\beta}(\rd x)$, where $c=1$, $\beta=1$, and several choices for $\alpha$ and $p$.

\begin{figure}
\centering
\includegraphics[width=16cm, height=6cm]{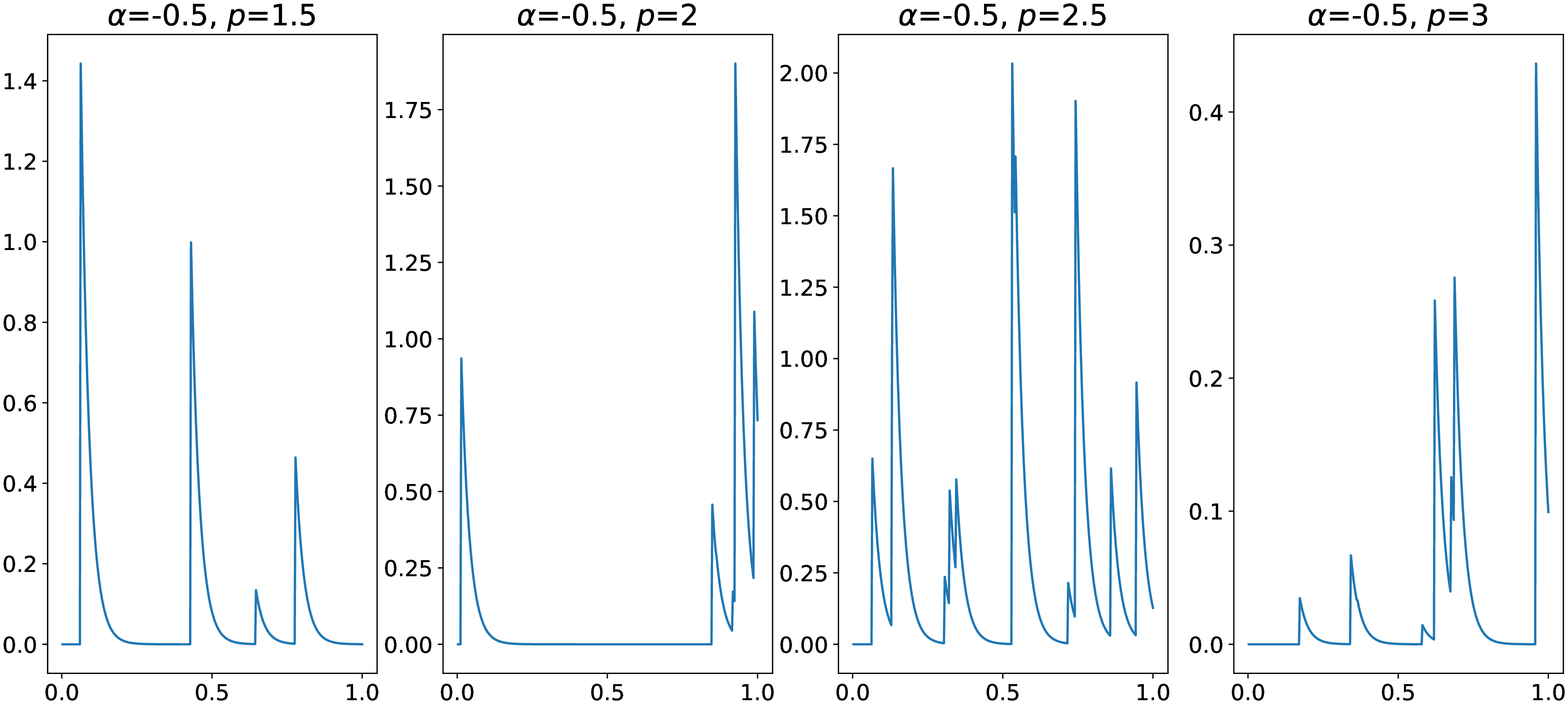}
\caption{Sample trajectories of OUTS processes with parameter $\lambda=10$, initial value $y_0=0$, and time step $t=1/365$. The BDLP distribution is $\ts(\lambda R,0)$ with $R(\rd x) = c\beta^\alpha\delta_{1/\beta}(\rd x)$, where $c=0.1$, $\beta=1$, $\alpha=-0.5$, and $p\in\{1.5,2,2.5,3\}$.}
\label{fig:ptsou:neg:trajectories}
\end{figure}%

\section{Proofs}\label{sec: proofs}

In this section we give the proofs.

\begin{proof}[Proof of Lemma \ref{lemma: IGa is GGSM}]
Note that
\begin{eqnarray*}
f_{\beta,\gamma,p,\eta}(u) &=& \frac{u^{-1-\beta} }{K_{\beta,\gamma,p,\eta}\Gamma(\gamma)}\int_0^{u^p(\eta-1)} x^{\gamma-1}e^{-x-u^p}\rd x\\
&=& \frac{u^{p\gamma-\beta-1}p }{K_{\beta,\gamma,p,\eta}\Gamma(\gamma)}\int_1^{\eta^{1/p}} (\theta^p-1)^{\gamma-1}e^{-\theta^pu^p}\theta^{p-1}\rd \theta\\
&=& \int_1^{\eta^{1/p}} \left(p \frac{\theta^{p\gamma-\beta}u^{p\gamma-\beta-1} }{\Gamma(\gamma-\beta/p)} e^{-(\theta u)^p}\right)\left(\frac{p}{K^*_{\beta,\gamma,p,\eta}}(\theta^p-1)^{\gamma-1}\theta^{p+\beta-p\gamma-1}\right)\rd \theta,
\end{eqnarray*}
where the second line follows by the change of variables $\theta^p=(xu^{-p}+1)$. 
\end{proof}

\begin{proof}[Proof of Proposition \ref{prop: moments of IBGM}.]
First, let $Y\sim \GGa(p\gamma -\beta, p, \theta^p)$ and note that 
$$
\rE[Y^\xi] =\frac{\Gamma\left(\gamma+(\xi-\beta)/p\right)}{\Gamma\left(\gamma-\beta/p\right)}\theta^{-\xi}.
$$
From here, by a conditioning argument, we have 
\begin{align}
\rE[W^\xi] &= \frac{\Gamma\left(\gamma+\xi-\beta)/p\right)}{\Gamma\left(\gamma-\beta/p\right)} \int_1^\eta \theta^{-\xi} m^\sharp_{\beta,\gamma,p,\eta}(\theta)\rd \theta\notag\\
&= \frac{\Gamma\left(\gamma+(\xi-\beta)/p\right)}{\Gamma\left(\gamma-\beta/p\right) C^*_{\beta,\gamma,p,\eta}}
\sum_{k=0}^{\gamma-1} {\gamma-1\choose k} \frac{ (-1)^{k}}{pk-\beta} \int_1^\eta\theta^{-\xi-1}\left(1-\theta^{\beta-pk}\right) \rd\theta,  \label{eq:moments:new:dist}
\end{align}
where
$$
\int_1^\eta\theta^{-\xi-1}\left(1-\theta^{\beta-pk}\right) \rd\theta=\frac{1-\eta^{-\xi}}{\xi} + \frac{1-\eta^{\beta-pk-\xi} }{\beta-pk-\xi}.
$$
Next, using the fact that 
$$
\int_1^\eta \theta^{-\xi-1}\ln\theta\rd \theta = \frac{1-\eta^{-\xi}(\xi\ln\eta +1)}{\xi^2},
$$
the result can be proved in a similar way if $\beta=pk$ for some $k\in\{0,1,2,\dots,\gamma-1\}$.
\end{proof}

Theorem \ref{thrm: main bdlp} is an immediately consequence of the following lemma.

\begin{lemma}\label{lemma:second int}
In the context of Theorem \ref{thrm: main bdlp}, $Y$ is a Markov process with temporally homogenous transition function $P_t(y,\rd x)$ 
having characteristic function
$\int_{\mathbb R^d}e^{i\langle x,z\rangle} P_t(y,\rd x) = \exp\left\{C_t(y,z)\right\}$, where
\begin{eqnarray*}
C_t(y,z) &=& ie^{-\lambda t}\langle y,z\rangle +  i   \left(1-e^{-\lambda t}\right) \langle  b,z\rangle - \sum_{n=0}^{\gamma-1} i\langle b_n,z\rangle\\
&&\quad + \int_{\mathbb R^d}\int_0^\infty  \psi_\alpha(e^{-\lambda t}z,xv)e^{-v^p} v^{-1-\alpha} \rd v R^*_0(\rd x)\\
&&\quad + \sum_{n=1}^{\gamma-1} \int_{\mathbb R^d}\int_0^\infty  \psi_{\alpha-np}(e^{-\lambda t}z,xv)  e^{-v^p}v^{np-1-\alpha} \rd v R^*_n(\rd x)\\
&&\quad +  \frac{pC_{\alpha,\gamma,p,e^{\lambda t}} R(\mathbb R^d)}{(\gamma-1)!}  \int_{\mathbb R^d}\int_0^\infty  \psi_0(z,xv)f^\sharp_{\alpha,\gamma,p,e^{\lambda t}}(v) \rd v R^1(\rd x)
\end{eqnarray*}
and
\begin{eqnarray*}
\psi_\alpha(z,x) = e^{i\langle z,x \rangle} - 1 - i\langle z,x \rangle 1_{[\alpha\ge1]}.
\end{eqnarray*}
\end{lemma}

\begin{proof}
Proposition 2.13 in \cite{Rocha-Arteaga:Sato:2019} implies that
\begin{eqnarray*}
C_t(y,z) =  ie^{-\lambda t}\langle y,z\rangle +  i   \left(1-e^{-\lambda t}\right) \langle  b,z\rangle +  \lambda \int_0^t  \int_{\mathbb R^d}\psi_\alpha(e^{-\lambda s}z,x) M(\rd x)\rd s,
\end{eqnarray*}
where $M$ is the L\'evy measue of $\ts(R,b)$. Now using \eqref{eq: Levy measure of TS} and the fact that $\psi_\alpha(az,x)=\psi_\alpha(z,ax)$ for any $a\in\mathbb R$ gives
\begin{eqnarray}
&&\lambda\int_0^t  \int_{\mathbb R^d}\psi_\alpha(e^{-\lambda s}z,x) M(\rd x)\rd s\nonumber \\
&&\qquad = \lambda\int_{\mathbb R^d}\int_0^\infty \int_0^t  \psi_\alpha(z,xue^{-\lambda s})u^{-1-\alpha} e^{-u^p}\rd s \rd uR(\rd x) \nonumber \\
&&\qquad = \int_{\mathbb R^d}\int_0^\infty \int_{ue^{-\lambda t}}^u  \psi_\alpha(z,xv)u^{-1-\alpha} e^{-u^p}v^{-1}\rd v \rd uR(\rd x) \nonumber \\
&&\qquad = \int_{\mathbb R^d}\int_0^\infty  \psi_\alpha(z,xv)v^{-1}\int_{v}^{ve^{\lambda t}}  u^{-1-\alpha} e^{-u^p} \rd u \rd v R(\rd x).\label{eq: for proof OUTS}
\end{eqnarray}

Note that for $v>0$ we have
\begin{eqnarray*}
&&v^{-1}\int_v^{ve^{\lambda t}} u^{-1-\alpha} e^{-u^p}\rd u = \int_1^{e^{\lambda t}} (uv)^{-1-\alpha} e^{-(vu)^p}\rd u\\
&&\quad= e^{-v^pe^{\lambda tp}}\int_1^{e^{\lambda t}} (uv)^{-1-\alpha} \rd u + \int_1^{e^{\lambda t}} \left(e^{-(vu)^p}- e^{-v^pe^{\lambda tp}}\right)(uv)^{-1-\alpha} \rd u\\
&&\quad= \frac{e^{-v^pe^{\lambda tp}}}{\alpha}\left(1-{e^{-\alpha\lambda t}}\right) v^{-1-\alpha}+  e^{-v^pe^{\lambda tp}} \sum_{n=1}^{\gamma-1} v^{np-1-\alpha}\int_1^{e^{\lambda t}}\frac{(e^{\lambda tp}-u^p)^n}{n!} u^{-1-\alpha} \rd u\\
&&\qquad+\int_1^{e^{\lambda t}} \left(e^{-(vu)^p} - e^{-v^pe^{\lambda tp}}\sum_{n=0}^{\gamma-1}\frac{(e^{\lambda tp}-u^p)^n}{n!}v^{np}\right)(uv)^{-1-\alpha} \rd u.
\end{eqnarray*}

Now applying Lemma 1 in \cite{Grabchak21}
\begin{eqnarray*}
&& \int_1^{e^{\lambda t}} \left(e^{-(vu)^p} - e^{-v^pe^{\lambda tp}}\sum_{n=0}^{\gamma-1}\frac{(e^{\lambda tp}-u^p)^n}{n!}v^{np}\right)(uv)^{-1-\alpha} \rd u \\
&&\qquad = \frac{1}{(\gamma-1)!} \int_1^{e^{\lambda t}} \int_0^{v^p(e^{\lambda tp}-u^p)} e^{-(x+v^pu^p)}x^{\gamma-1}\rd x (uv)^{-1-\alpha} \rd u \\
&&\qquad = \frac{p v^{p\gamma-\alpha-1}}{(\gamma-1)!} \int_1^{e^{\lambda t}} \int_u^{e^{\lambda t}} e^{-v^py^p}(y^p-u^p)^{\gamma-1}y^{p-1}\rd y u^{-1-\alpha} \rd u \\
&&\qquad = \frac{p v^{p\gamma-\alpha-1}}{(\gamma-1)!} \int_1^{e^{\lambda t}} y^{p-1} e^{-v^py^p} \int_1^y (y^p-u^p)^{\gamma-1} u^{-1-\alpha} \rd u \rd y\\
&&\qquad = \frac{p v^{p\gamma-\alpha-1}}{(\gamma-1)!} \int_1^{e^{\lambda t}} y^{p\gamma-\alpha-1} e^{-v^py^p} \int_{1/y}^1 (1-s^p)^{\gamma-1} s^{-1-\alpha} \rd s \rd y\\
&&\qquad = \frac{pC_{\alpha,\gamma,p,e^{\lambda t}} }{(\gamma-1)!} f^\sharp_{\alpha,\gamma,p,e^{\lambda t}}(v),
\end{eqnarray*}
where the third line follows by the substitution $y^p=v^{-p}x+u^p$ and the fifth by the substitution $s=u/y$. Now putting everything together and using the readily checked facts that
$$
\int_0^\infty  \psi_\alpha(z,xv)  e^{-v^pe^{\lambda tp}}v^{np-1-\alpha} \rd v = e^{\lambda t(\alpha-np)} \int_0^\infty  \psi_\alpha(z,e^{-\lambda t}xv)  e^{-v^p}v^{np-1-\alpha} \rd v
$$
and
$$
\int_1^{e^{\lambda t}}\frac{(e^{\lambda tp}-u^p)^n}{n!} u^{-1-\alpha} \rd u = e^{\lambda t(pn-\alpha)} \int^1_{e^{-\lambda t}}\frac{(1-u^p)^n}{n!} u^{-1-\alpha} \rd u
$$
gives the result.
\end{proof}

\begin{proof}[Proof of Theorem \ref{thrm: main bdlp alpha neg}]
Following the proof of Lemma \ref{lemma:second int} to \eqref{eq: for proof OUTS} shows that the characteristic function of the temporally homogenous transition function $P_t(y,\rd x)$ 
is given by 
$\int_{\mathbb R^d}e^{i\langle x,z\rangle} P_t(y,\rd x) = \exp\left\{C_t(y,z)\right\}$, where
\begin{eqnarray*}
C_t(y,z) &=& ie^{-\lambda t}\langle y,z\rangle +  i   \left(1-e^{-\lambda t}\right) \langle  b,z\rangle \\
&&\quad +\int_{\mathbb R^d}\int_0^\infty  \psi_\alpha(z,xv)v^{-1}\int_{v}^{ve^{\lambda t}}  u^{-1-\alpha} e^{-u^p} \rd u \rd v R(\rd x).
\end{eqnarray*}
Noting that the formula on the second line equals
$$
p^{-1}\lambda t \Gamma(|\alpha|/p)R(\mathbb R^d)\int_{\mathbb R^d}\int_0^\infty  \psi_\alpha(z,xv) h_{|\alpha|,p,e^{\lambda t}}(v) \rd v R^1(\rd x)
$$
gives the result.
\end{proof}

\end{document}